\documentclass{amsart}
\usepackage[english]{babel}
\usepackage{amssymb,amsmath,hyperref}
\usepackage{bbold}
\usepackage{amsrefs}
\usepackage{stackrel}
\usepackage{mathrsfs}

\newtheorem{thm}{Theorem}

\newtheorem{cor}[thm]{Corollary}
\newtheorem{defi}[thm]{Definition}
\newtheorem{rem}[thm]{Remark}
\newtheorem{nota}[thm]{Notation}
\newtheorem{exa}[thm]{Example}

\newtheorem{des}[thm]{Description}
\newtheorem{ack}[thm]{Acknowledgement}

\newcommand\be{\begin{equation}}
\newcommand\ee{\end{equation}}

\newcommand{\tup}{\underline} 

\def\bdefi{\begin{defi}\rm}
\def\edefi{\end{defi}}
\def\bnota{\begin{nota}\rm}
\def\enota{\end{nota}}
\def\brem{\begin{rem}\rm}
\def\erem{\end{rem}}

\def\FIVE{\Pi_{1}^{1}\text{-\textsf{CA}}_{0}}

\newcommand{\Sh}{\ensuremath{\protect{S_{\st{}}}}}

\newcommand{\intern}{\textup{{\textsf{int}}}}
\newcommand{\existsst}{\exists^{\st{}}}
\newcommand{\forallst}{\forall^{\st{}}}

\def\ATR{\textup{\textsf{ATR}}}
\def\NCR{\textup{\textsf{NCR}}}
\def\IST{\textup{\textsf{IST}}}
\def\ZFC{\textup{\textsf{ZFC}}}

\def\H{\textup{\textsf{H}}}
\def\RCA{\textup{\textsf{RCA}}}

\def\RCAo{\textup{\textsf{RCA}}_{0}^{\omega}}

\def\ef{\textup{\textsf{ef}}}
\def\ns{\textup{\textsf{ns}}}

\def\WKL{\textup{\textsf{WKL}}}

\def\T{\mathcal{T}}

\def\bye{\end{document}}

\def\P{\textup{\textsf{P}}}

\def\N{{\mathbb  N}}

\def\R{{\mathbb  R}}

\def\I{{\textsf{\textup{I}}}}

\def\MCT{\textup{\textsf{MCT}}}

\def\R{{\mathbb{R}}}
\def\({\textup{(}}
\def\){\textup{)}}

\def\st{\textup{st}}

\def\asa{\leftrightarrow}

\def\di{\rightarrow}

\def\eps{\varepsilon}

\def\M{\mathcal{M}}

\def\ACA{\textup{\textsf{ACA}}}
\def\paai{\Pi_{1}^{0}\textup{-\textsf{TRANS}}}
\def\Paai{\Pi_{1}^{1}\textup{-\textsf{TRANS}}}

\def\QFAC{\textup{\textsf{QF-AC}}}

\def\pw{\textup{\textsf{pw}}}

\def\DIV{\textup{\textsf{DIV}}}

\def\PST{\textup{\textsf{PST}}}

\def\RIE{\textup{\textsf{RIE}}}

\def\MU{\textup{\textsf{MU}}}
\def\MUO{\textup{\textsf{MUO}}}

\def\HAC{\textup{\textsf{HAC}}}

\def\INT{\textup{\textsf{int}}}

\numberwithin{equation}{section}
\numberwithin{thm}{section}

\usepackage{comment}
\begin{document}
\title{Reverse Formalism 16}

\author{Sam Sanders}
\address{Munich Center for Mathematical Philosophy, LMU Munich, Germany \& Department of Mathematics, Ghent University} 
\email{sasander@me.com}
\maketitle
\begin{abstract}
In his remarkable paper \emph{Formalism 64}, Robinson defends his eponymous position concerning the foundations of mathematics, as follows:
\begin{enumerate}
\item[(i)]  Any mention of infinite totalities is literally meaningless.  
\item[(ii)] We should act as if infinite totalities really existed.  
\end{enumerate}
Being the originator of \emph{Nonstandard Analysis}, it stands to reason that Robinson would have often been faced with the opposing position that `some infinite totalities are more meaningful than others', the textbook example being that of infinitesimals (versus less controversial infinite totalities).  
For instance, Bishop and Connes have made such claims regarding infinitesimals, and Nonstandard Analysis in general, going as far as calling the latter respectively a \emph{debasement of meaning} and \emph{virtual}, while accepting as meaningful other infinite totalities and the associated mathematical framework.   

\medskip

We shall study the critique of Nonstandard Analysis by Bishop and Connes, and observe that these authors equate `meaning' and `computational content', though their interpretations of said content vary.  As we will see, Bishop and Connes claim that the presence of ideal objects (in particular infinitesimals) in Nonstandard Analysis yields the absence of meaning (i.e.\ computational content).  
We will debunk the Bishop-Connes critique by establishing the contrary, namely that the presence of ideal objects (in particular infinitesimals) in Nonstandard Analysis yields the \emph{ubiquitous presence} of computational content.  In particular, infinitesimals provide an \emph{elegant shorthand} for expressing computational content.       
To this end, we introduce a direct translation between a large class of theorems of Nonstandard Analysis and theorems rich in computational content (not involving Nonstandard Analysis), similar to the `reversals' from the foundational program \emph{Reverse Mathematics}.  
The latter also plays an important role in gauging the scope of this translation.  
\end{abstract}

\medskip

Accepted for publication in \emph{Synthese}, special issue on the foundations of mathematics (eds.\ John Wigglesworth, Carolin Antos-Kuby, Neil Barton, Sy David Friedman, and Claudio Ternullo), 2017.
\section{Reverse formalism: What's in a name?}
Implicit in the title of this paper is Robinson's remarkable paper \emph{Formalism 64}  (\cite{robinson64}) in which he renounces his previous platonist philosophy of mathematics and outlines his newfound\footnote{
Note that Robinson still held the formalist view nearly ten years after Formalism 64.  
\begin{quote}
Subsequent exchanges, both oral and in published writings, have not induced me to change my views [from Formalism 64]. Moreover, I
believed then and I still believe that the well-known recent developments in set theory represent evidence favoring these views. 
(\cite{roku}*{p.\ 42})
\end{quote}
Moreover, Dauben discusses the origin of Robinson's `change of heart' in \cite{tauben}*{\S5}.} formalist beliefs:  
\begin{enumerate}
\item[(i)]  Any mention of infinite totalities is literally meaningless.  
\item[(ii)] We should act as if infinite totalities really existed.  
\end{enumerate}
Most anti-realist positions in the philosophy of mathematics do not go quite as far:  While parts of mathematics 
are rejected as meaningless -constructivism and the law of excluded middle are discussed in Section \ref{hitch} below- other parts are accepted as meaningful (usually only after careful justification).  

\medskip

By way of an \emph{infamous} example of such rejection, Bishop went as far as debasing Nonstandard Analysis as a \emph{debasement of meaning} (\cite{kluut}*{p.\ 513}), while Connes has referred to Nonstandard Analysis as \emph{virtual} and \emph{a chimera} (\cite{kano2}*{\S3.1}).   
In particular, for rather different reasons and in different contexts, Bishop and Connes equate `meaningful mathematics' and `mathematics with computational content', and therefore claim {Nonstandard Analysis} is devoid of meaning as it lacks -in their view- any and all computational content.  

\medskip

As may be expected in light of their different philosophies of mathematics\footnote{See Section \ref{hitch} (resp.\ Section \ref{koko}) for a discussion of Bishop's (resp.\ Connes') philosophy of mathematics.}, Connes and Bishop have diverging opinions on what constitutes `computational content', but their motivation for claiming Nonstandard Analysis' lack thereof is the same, namely as follows:
\begin{center}
\emph{The presence of ideal objects \(in particular infinitesimals\) in Nonstandard Analysis yields the absence of computational content.}
\end{center}
We shall refer to this claim as the \emph{Bishop-Connes critique} (of Nonstandard Analysis).
As discussed in Section \ref{PO}, Connes has formulated this critique in print more or less literally, while it is implicit in Bishop's philosophy of mathematics.  

\medskip

Furthermore, the \emph{arguments by Connes and Bishop for} this critique have been dissected in surprising detail in various places \emph{and found wanting} (as also discussed in Section \ref{PO}).  
Nonetheless, there has been no \emph{full-scale debunking} of the Bishop-Connes critique.  
The aim of this paper precisely is to debunk this critique by establishing the following `opposite' claim:
\begin{center}
\emph{The presence of ideal objects \(in particular infinitesimals\) in Nonstandard Analysis yields the \emph{ubiquitous presence} of computational content.}
\end{center}
 In particular, we shall show that infinitesimals provide an \emph{elegant shorthand} for expressing computational content. 
To this end, we establish a direct translation between (proofs of) theorems of Nonstandard Analysis and (proofs of) theorems rich in computational content, similar to the `reversals' from the foundational program \emph{Reverse Mathematics}.  
The latter also plays an important role in gauging the scope of this translation.  
Finally, the translation at hand is merely a syntactic manipulation of finite objects (proofs) given by an algorithm (in the sense of Bishop), as discussed in Section \ref{promi}.  

\medskip

We label the aforementioned translation as `reverse formalism', as it bestows Bishop-Connes style meaning, i.e.\ computational content, onto Nonstandard Analysis, and is similar to the reversals in Reverse Mathematics (See Section \ref{RM}).  
As to the further structure of this paper, we shall briefly discuss Nonstandard Analysis in Section \ref{NIST}.
The foundational program Reverse Mathematics is introduced in Section \ref{RM}. 
The views of Bishop and Connes on Nonstandard Analysis are discussed in Section~\ref{PO}, while the main contribution of this paper, reverse formalism and the aforementioned translation, is discussed in Section~\ref{RF}.  

\section{Reverse Mathematics}\label{RM}
We shall introduce the program \emph{Reverse Mathematics}, sketch its main results, and discuss the associated vague notion of `mathematical theorem'.  
As we will see, the classification provided by Reverse Mathematics is quite elegant \emph{but inherently vague}.  In particularly, there is no meta-theorem or 
formula class capturing this classification.  Our results in Section \ref{RF} are based on Reverse Mathematics in the sense that we obtain a similar classification \emph{and that our classification is similarly inherently vague}.  In our opinion, a lot of philosophical insight into mathematics can be reaped from the study of these vague aspects of the discipline, but that is beyond the scope of the current paper.       
    
\subsection{Introducing Reverse Mathematics}\label{RM!}
Reverse Mathematics (RM) is a program in the foundations of mathematics initiated around 1975 by Friedman (\cites{fried,fried2}) and developed extensively by Simpson (\cite{simpson2, simpson1}) and others.  
The aim of RM is to find the axioms \emph{necessary} to prove a statement of \emph{ordinary} mathematics, i.e.\ dealing with countable or separable objects.   
Simpson neatly summarises the main results:
\begin{quote}
In many cases, if a mathematical theorem is proved from appropriately weak set existence axioms, then the axioms will be logically equivalent to the theorem. Furthermore, only a few specific set existence axioms arise repeatedly in this context, which in turn correspond to classical foundational programs. This is the theme of reverse mathematics, [\dots] (\cite{simpson2}*{Preface to the second edition}).
\end{quote}
We now discuss these results in more detail while referring to Simpson's monograph \cite{simpson2} for full details like the exact definitions of the formal systems used.  

\medskip

First of all, the classical\footnote{In \emph{Constructive Reverse Mathematics} (\cite{ishi1}), intuitionistic logic is used instead.} base theory $\RCA_{0}$ of `computable\footnote{$\RCA_{0}$ consists of induction $I\Sigma_{1}$, and the {\bf r}ecursive {\bf c}omprehension {\bf a}xiom $\Delta_{1}^{0}$-CA.} mathematics' is usually assumed to be given in RM.  
Thus, the aim of RM is as follows:  
\begin{quote}
\emph{The aim of \emph{RM} is to find the minimal axioms $A$ such that $\RCA_{0}$ proves $ [A\di T]$ for statements $T$ of ordinary mathematics.}
\end{quote}
Surprisingly, once the minimal axioms $A$ have been found, we almost always also have $\RCA_{0}\vdash [A\asa T]$, i.e.\ not only can we derive the theorem $T$ from the axioms $A$ (the `usual' way of doing mathematics), we can also derive the axiom $A$ from the theorem $T$ (the `reverse' way of doing mathematics).  In light of the latter, the field was baptised `Reverse Mathematics'.    

\medskip

Secondly, and perhaps even more surprisingly, in the majority\footnote{Exceptions are classified in the so-called Reverse Mathematics Zoo (\cite{damirzoo}).  Most of these are `combinatorial in nature', another vague notion.} 
of cases for a statement $T$ of ordinary mathematics, either $T$ is provable in $\RCA_{0}$, or the latter proves $T\asa A_{i}$, where $A_{i}$ is one of the logical systems $\WKL_{0}, \ACA_{0},$ $ \ATR_{0}$ or $\FIVE$.  The latter together with $\RCA_{0}$ form the `Big Five' and the aforementioned observation that most mathematical theorems fall into one of the Big Five categories, is called the \emph{Big Five phenomenon} (\cite{montahue}*{p.~432}).  
Furthermore, each of the Big Five has a natural formulation in terms of (Turing) computability (See e.g.\ \cite{simpson2}*{I.3.4, I.5.4, I.7.5}).
As noted by Simpson in \cite{simpson2}*{I.12}, each of the Big Five also corresponds (sometimes loosely) to a foundational program in mathematics.  

\subsection{Ordinary mathematics and other vague notions}\label{hotgie}
A crucial point regarding RM is that the two main results from Section \ref{RM!} (namely `reversals' and the `Big Five phenomenon') are \emph{heuristic and qualitative} observations following the empirical study of theorems of ordinary mathematics.  Obviously, the category `ordinary mathematics' is vague, but is described by Simpson as well as can be expected, namely as follows:
\begin{quote}
We identify as \emph{ordinary} or \emph{non-set-theoretic} that body of mathematics which is prior to or independent of the introduction of abstract set-theoretic concepts. We have in mind such branches as geometry, number theory, calculus, differential equations, real and complex analysis, countable algebra, the topology of complete separable metric spaces, mathematical logic, and computability theory. (Emphasis in original, \cite{simpson2}*{p.\ 1})
\end{quote}
However, the vagueness in the above main results of RM runs much deeper:  Since `ordinary mathematics' does not have a formal definition, how should we then understand the notion `theorem of ordinary mathematics'?  In particular, when and why does one bestow the title `theorem' onto an arbitrary sentence provable in second-order arithmetic?

\medskip

While we do not claim to answer this question, the notion `theorem of ordinary mathematics' can be elucidated as follows:  When formalising mathematics  in second-order arithmetic $\textsf{Z}_{2}$ as in RM, it becomes apparent that \emph{only a small fragment of the $\textsf{\textup{Z}}_{2}$} is needed.  In particular, the strongest Big Five system $\FIVE$ suffices for formalising almost all of ordinary mathematics, while this system is only the `first fragment' of $\textsf{Z}_{2}=\cup_{k}\Pi_{k}^{1}\textsf{-CA}_{0}$.  Furthermore, the Big Five systems of RM were initially formulated with the axiom schema of induction \emph{for any formula} (\cite{fried}*{p.\ 236}), but it was soon realised that induction is \emph{only needed for purely existential formulas}.  The subscript `$0$' in the Big Five systems refers to this use of restricted induction.  

\medskip

Hence, we observe that large parts of the induction and comprehension axioms in $\textsf{Z}_{2}$ are not needed for formalising the known results of ordinary mathematics (the latter as described in the above quote by Simpson).  
It is a natural first step to identify these `unneeded' axioms as `non-mathematical'.  While this falls short of providing a formal definition for `theorem of ordinary mathematics', we believe that it is essential to make a distinction between sentences of second-order arithmetic which are `mathematical' and which are `non-mathematical' in nature.  The exact distinction between these two concepts is inherently vague, but it is a `real' distinction, as evidenced by the aforementioned existence of `unneeded' axioms which make up the bulk of $\textsf{Z}_{2}$.  

\medskip

So far, we have left unanswered the question of when and why one bestows the title `theorem' onto an arbitrary sentence provable in second-order arithmetic.  We \emph{have} observed that some sentences of second-order arithmetic may be called `non-mathematical' in the sense that they are not needed in the formalisation of mathematics in $\textsf{Z}_{2}$. Thus, we have established some evidence for the reality of the distinction `mathematical versus non-mathematical sentence' in second-order arithmetic. The notion of `theorem of ordinary mathematics' has therewith become slightly more real.   

\medskip

Finally, the above observations are not limited to RM alone:  There are at least two other fields where a similar `mathematical versus non-mathematical' distinction is made.  Firstly, Bishop makes a clear distinction between `real mathematics' and formal systems in the following quotes from \cite{bish1}*{p.\ 6}.
\begin{quote}
A bugaboo of both Brouwer
and the logicians has been compulsive speculation about the nature
of the continuum. In the case of the logicians this leads to 
contortions in which various formal systems, all detached from reality,
are interpreted within one another in the hope that the nature of the
continuum will somehow emerge. 
\end{quote}
\begin{quote}
In fairness to Brouwer it should
be said that he did not associate himself with these efforts to formalize
reality; it is the fault of the logicians that many mathematicians who
think they know something of the constructive point of view have in
mind a dinky formal system or, just as bad, confuse constructivism
with recursive function theory.  \end{quote}
It should be noted that Bishop revised his views on the value of formal systems later in life (See e.g.\ \cite{nukino}*{p.\ 60}), as also discussed in Section \ref{promi}.

\medskip

Secondly, ever since G\"odel's famous incompleteness theorems (See e.g.\ \cite{buss}*{II}), it is known that a reasonably rich and consistent logical system cannot prove its own consistency.  
For many years, it was then an open problem to find \emph{mathematically natural} examples of statements not provable in e.g.\ Peano arithmetic.  Paris and Harrington succeeded in finding
such a statement, as suggested by the quote:
\begin{quote}
We investigate a reasonably natural theorem of finitary combinatorics, a simple
extension of the Finite Ramsey Theorem. This chapter is mainly devoted to
demonstrating that this theorem, while true, is not provable in Peano
arithmetic. (\cite{barwise}*{p.\ 1134})
\end{quote}
The above examples suggest that the `mathematical versus non-mathematical' distinction is real at least in the sense of 
being commonplace to people working in mathematical logic.  As noted above, a lot of philosophical insight into mathematics can be reaped -in our opinion- from the study of these vague aspects of mathematical logic, but that is beyond the scope of this paper.        

\section{Nonstandard Analysis}\label{NIST}
We introduce Nelson's \emph{internal set theory} $\IST$, a well-known axiomatic approach to Nonstandard Analysis first introduced in \cite{wownelly}.
An introduction to the \emph{practice} of $\IST$ may be found in \cite{iamrobert}.     
We discuss certain important fragments of $\IST$ from \cite{brie} in Section \ref{frag}.  
\subsection{Nelson's axiomatic approach to Nonstandard Analysis}\label{IST}
In Nelson's \emph{syntactic} (or `axiomatic') approach to Nonstandard Analysis (\cite{wownelly}), as opposed to Robinson's semantic one (\cite{robinson1}), a new predicate `st($x$)', read as `$x$ is standard' is added to the language of \textsf{ZFC}, the usual foundation of mathematics\footnote{The acronym $\ZFC$ stands for \emph{Zermelo-Fraenkel set theory with the axiom of choice}; see \cite{jech} for an introduction to set theory.}.  
The notations $(\forall^{\st}x)$ and $(\exists^{\st}y)$ are short for $(\forall x)(\st(x)\di \dots)$ and $(\exists y)(\st(y)\wedge \dots)$.  A formula is called \emph{internal} if it does not involve `st', and \emph{external} otherwise.   
The external axioms \emph{Idealisation}, \emph{Standardization}, and \emph{Transfer} govern the new predicate `st';  They are respectively defined\footnote{The superscript `fin' in \textsf{(I)} means that $x$ is finite, i.e.\ its number of elements are bounded by a natural number.} as:  
\begin{enumerate}
\item[\textsf{(I)}] $(\forall^{\st~\textup{fin}}x)(\exists y)(\forall z\in x)\varphi(z,y)\di (\exists y)(\forall^{\st}x)\varphi(x,y)$, for internal $\varphi$ with any (possibly nonstandard) parameters.  
\item[\textsf{(S)}] $(\forall^{\st} x)(\exists^{\st}y)(\forall^{\st}z)\big((z\in x\wedge \varphi(z))\asa z\in y\big)$, for any $\varphi$.
\item[\textsf{(T)}] $(\forall^{\st}t)\big[(\forall^{\st}x)\varphi(x, t)\di (\forall x)\varphi(x, t)\big]$, where $\varphi(x,t)$ is internal, and only has free variables $t, x$.  
\end{enumerate}
The system \textsf{IST} is the \emph{internal} system \textsf{ZFC} extended with the aforementioned \emph{external} axioms;  
The former is a conservative extension of \textsf{ZFC} for the internal language (\cite{wownelly}*{\S8}).  It goes without saying that 
the above extension may be done for a large spectrum of logical systems other than $\ZFC$.  In Section \ref{frag}, we study this extension of the usual axiomatisation of 
arithmetic, called \emph{Peano Arithmetic}.  As noted by Nelson, we can think of the standard sets of $\IST$ as those sets of mainstream mathematics.   
\begin{quote}
\emph{Every specific object of conventional mathematics is a standard set}. It remains unchanged in the new theory $[\IST]$. (\cite{wownelly}*{p.\ 1166})
\end{quote}
We shall use `mainstream' and `standard' mathematics interchangeably in the context of $\IST$.  
Finally, we discuss the intuitive meaning of the external axioms.  

\medskip

First of all, the \emph{contraposition} of \textsf{I} implies that for all internal $\varphi$:
\be\label{criv}
(\forall y)(\exists^{\st}x)\varphi(x,y) \di (\exists^{\st~\textup{fin}}x)\underline{(\forall y)(\exists z\in x)\varphi(z,y)},
\ee
where the underlined part is \emph{internal}.  Hence, intuitively speaking, \emph{Idealisation} allows us to `pull a standard quantifiers like $(\exists^{\st}x)$ in \eqref{criv} through a normal quantifier $(\forall y)$'.  Note that the axioms $B\Sigma_{n}$ of Peano arithmetic (\cite{buss}*{II}) play a similar role:  The former allow one to `pull an unbounded number quantifier through a bounded number quantifier'.  In each case, one obtains a formula in a kind of `normal form' with a block of certain quantifiers (resp.\ external/unbounded) up front followed by another block of different quantifiers (resp.\ internal/bounded).  The following example involving nonstandard continuity in $\IST$ is illustrative (See \cite{wownelly}*{\S5} for more examples).   
\begin{rem}\label{leffe}\rm
We say that $f$ is \emph{nonstandard continuous} on the set $X\subseteq \R$ if
\be\label{soareyouke}
(\forall^{\st}x\in X)(\forall y\in X)(x\approx y \di f(x)\approx f(y)),  
\ee
where $z\approx w$ is $(\forall^{\st} n\in \N)(|z-w|<\frac{1}{n})$.  Resolving `$\approx$' in \eqref{soareyouke}, we obtain  
\[\textstyle
(\forall^{\st}x\in X)(\forall y\in X)\big( (\forall^{\st}N\in\N) (|x- y|<\frac{1}{N}) \di (\forall^{\st}k \in\N)(|f(x)- f(y)|<\frac{1}{k})\big).  
\]
Using classical logic, we may bring out the `$(\forall^{\st}k\in \N)$' and $(\forall^{\st}N\in\N) $ quantifiers as follows:
\[\textstyle
(\forall^{\st}x\in X)(\forall^{\st}k\in \N)\underline{(\forall y\in X)(\exists^{\st}N\in\N)\big( |x- y|<\frac{1}{N} \di|f(x)- f(y)|<\frac{1}{k}\big)}.  
\]
Applying \textsf{I} as in \eqref{criv} to the underlined formula, we obtain a finite and standard set $z\subset \N$ such that $(\forall y\in X)(\exists N\in z)$ in the previous formula.  
Now let $N_{0}$ be the maximum of all numbers in $z$, and note that for $N=N_{0}$
\[\textstyle
(\forall^{\st}x\in X)(\forall^{\st}k\in \N)(\exists^{\st}N\in\N)(\forall y\in X)\big( |x- y|<\frac{1}{N} \di|f(x)- f(y)|<\frac{1}{k}\big).  
\]
The previous formula has all standard quantifiers up front and is very close to the `epsilon-delta' definition of continuity from mainstream mathematics.  
Hence, we observe the role of \textsf{I}: to connect the worlds of nonstandard mathematics (as in \eqref{soareyouke}) and mainstream mathematics.    
\end{rem}
Secondly, the axiom \emph{Transfer} expresses that certain statements about \emph{standard} objects are also true for \emph{all} objects.    This property is essential in proving the equivalence between so-called epsilon-delta statements and their nonstandard formulation.  The following example involving continuity is illustrative.
\begin{exa}\rm
Recall Example \ref{leffe}, the definition of nonstandard continuity \eqref{soareyouke} and the final equation in particular.  The latter yields the following by dropping the `st' for $N$:
\[\textstyle
(\forall^{\st}x\in X)(\forall^{\st}k\in \N)(\exists N\in\N)(\forall y\in X)\big( |x- y|<\frac{1}{N} \di|f(x)- f(y)|<\frac{1}{k}\big).  
\]
Assuming $X$ and $f$ to be standard, we can apply \textsf{T} to the previous to obtain 
\be\textstyle\label{kikj}
(\forall x\in X)(\forall k\in \N)(\exists N\in\N)(\underline{\forall y\in X) |x- y|<\frac{1}{N} \di|f(x)- f(y)|<\frac{1}{k}}
\ee
Note that \eqref{kikj} is just the \emph{usual epsilon-delta definition of continuity}.  In turn, to prove that \eqref{kikj} implies nonstandard continuity as in \eqref{soareyouke}, fix \emph{standard} $X, f, k$ in \eqref{kikj} and apply the contraposition of \textsf{T} to `$(\exists N\in \N)\varphi(N)$' where $\varphi$ is the underlined formula in \eqref{kikj}.  
The resulting formula $(\exists^{\st} N\in \N)\varphi(N)$ immediately implies nonstandard continuity as in \eqref{soareyouke}.  
\end{exa}
By the previous example, nonstandard continuity \eqref{soareyouke} and epsilon delta continuity \eqref{kikj} are equivalent for standard functions in $\IST$.  However, the former involves far less quantifier alternations and is close to the intuitive understanding of continuity as `no jumps in the graph of the function'. 
Hence, we observe the role of \textsf{T}: to connect the worlds of nonstandard mathematics (as in \eqref{soareyouke}) and mainstream mathematics (as in \eqref{kikj}).     
         
\medskip

Thirdly, \emph{Standardization} (also called \emph{Standard Part}) is useful as follows: It is in general easy to build \emph{nonstandard and approximate} solutions to mathematical problems in $\IST$, but a \emph{standard} solution is needed as the latter also exists in `normal' mathematics (as it is suitable for \emph{Transfer}).   Intuitively, the axiom \textsf{S} tells us that from a \emph{nonstandard approximate solution}, we can \emph{always} find a \emph{standard one}.  Since we may apply \emph{Transfer} to formulas involving the latter, we can then also prove the latter is an object of normal mathematics.  The following example is highly illustrative. 
\begin{exa}\rm
The \emph{intermediate value theorem} states that for every continuous function $f:[0,1]\di \R$ such that $f(0) f(1)< 0$, there is $x\in [0,1]$ such that $f(x)=0$.  Assuming $f$ is standard, it is easy\footnote{By Example \ref{leffe}, we may assume $f$ is nonstandard continuous.  Let $N$ be a nonstandard natural number and let $j\leq N$ be the least number such that $f(\frac{j}{N})f(\frac{j+1}{N})\leq 0$.  Then $f(j/N)\approx 0$ by nonstandard continuity, and we are done. \label{Exake}} to find a nonstandard real $y$ in the unit interval such that $f(y)\approx 0$, i.e.\ $y$ is an intermediate value `up to infinitesimals'.  
The axiom \textsf{S} then tells\footnote{Let $y\in [0,1]$ be such that $f(y)\approx 0$ and consider the set of rationals $z=\{ q_{1}, q_{1}, q_{2}, \dots, q_{N}\}$ where $q_{i}$ is a rational such that $|y-q_{i}|<\frac{1}{i}$ and $q_{i}=\frac{j}{2^{i}}$ for some $j\leq 2^{i}$, and $N$ is a nonstandard number.  Applying \textsf{S}, there is a standard set $w$ such that $(\forall^{\st}i)(q_{i}\in w)$.  The standard sequence $q_{i}$ converges to a standard real $x\approx y$.} us that there is a \emph{standard} real $x$ such that $x\approx y$, and by the nonstandard continuity of $f$ (See previous example), we have $f(x)\approx 0$.  Now apply \textsf{T} to the latter\footnote{Recall that `$f(x)\approx 0$' is short for $(\forall^{\st}k)(|f(x)|<\frac{1}{k})$.} to obtain $f(x)=0$.  Hence, we have obtain the (internal) intermediate value theorem for standard functions, and \textsf{T} yields the full theorem.            
\end{exa}
Hence, we observe the role of \textsf{S}: to connect the worlds of nonstandard mathematics and mainstream (standard) mathematics by providing standard objects `close to' nonstandard ones.    

\medskip

In conclusion, the external axioms of $\IST$ provide a connection between nonstandard and mainstream mathematics:  They allow one to `jump back and forth' between the standard and nonstandard world.  This technique is useful as some problems (like switching limits and integrals) may be easier to solve in the discrete/finite world of nonstandard mathematics than in the continuous/infinite world of standard mathematics (or vice versa).  This observation lies at the heart of Nonstandard Analysis and is a first step towards understanding its power.       

\subsection{Fragments of Nelson's internal set theory}\label{frag}
Fragments of $\IST$ have been studied before and we are interested in the systems $\P$ and $\H$ introduced in \cite{brie}.
 In a nutshell, $\P$ and $\H$ are versions of $\IST$ based on the usual classical and intuitionistic axiomatisations of arithmetic, namely \emph{Peano and Heyting arithmetic}.  
We refer to \cite{kohlenbach3} for the exact definitions of our version of Peano and Heyting arithmetic, commonly abbreviated respectively as $\textsf{E-PA}^{\omega}$ and $\textsf{E-HA}^{\omega}$.  
In particular, the systems $\P$ and $\H$ are conservative\footnote{Like for $\ZFC$ and $\IST$, if the system $\P$ (resp.\ $\H$) proves an internal sentence, then this sentence is provable in $\textsf{E-PA}^{\omega}$ (resp.\ \textsf{E-HA}$^{\omega}$).} extensions of \emph{Peano arithmetic} $\textsf{E-PA}^{\omega}$ and \emph{Heyting arithmetic} $\textsf{E-HA}^{\omega}$, as also follows from Theorem \ref{TERM}.  We discuss the systems $\P$ and $\H$ in detail in Sections~\ref{graf1} and \ref{graf2}, and list them in full detail in Section \ref{FULL}.  We discuss the reason why $\P$ and $\H$ are important to our enterprise in Section \ref{fraki}.  

\subsubsection{The classical system $\P$}\label{graf1}
We discuss the fragment $\P$ of $\IST$ from \cite{brie}.   Similar to the way $\IST$ is an extension of $\ZFC$,
$\P$ is just the internal system $\textsf{E-PA}^{\omega}$ with the language extended with a new standardness predicate `\st' and with some special cases of the external axioms of $\IST$.  The technical details of this extension may be found in Section \ref{FULL} while we now provide an intuitive motivation for the external axioms of $\P$, assuming basic familiarity with the finite type system of G\"odel's system $T$ (also discussed in Section \ref{FULL}).  

\medskip

First of all, the system $\P$ does not include any fragment of \emph{Transfer}.  The motivation for this omission is as follows:  
The system $\ACA_{0}$ proves the existence of the non-computable \emph{Turing jump} (\cite{simpson2}*{III}) and very weak fragments of \textsf{T} already imply versions of $\ACA_{0}$.
In particular, the following axiom is the \emph{Transfer} axiom limited to universal number quantifiers.  
  \be\tag{$\paai$}
(\forall^{\st}f^{1})\big[(\forall^{\st}n^{0})f(n)\ne0\di (\forall m)f(m)\ne0\big].
\ee
As proved in \cite{sambon}*{\S4.1}, the system $\P+\paai$ proves the existence of the Turing jump, which is also implicit in the appendix to \cite{fega}.  
Thus, due to its non-constructive nature, no \emph{Transfer} is present in $\P$.    

\medskip

Secondly, the system $\P$ involves the full axiom \emph{Idealisation} in the language of $\P$ as follows:  For any internal formula in the language of $\P$:
\be\label{wellhung}
(\forall^{\st} x^{\sigma^{*}})(\exists y^{\tau} )(\forall z^{\sigma}\in x)\varphi(z,y)\di (\exists y^{\tau})(\forall^{\st} x^{\sigma})\varphi(x,y), 
\ee
As it turns out, the axiom \textsf{I} does not yield any `non-computable' consequences, which also follows from Theorem \ref{TERM}.  

\medskip

Thirdly, the system $\P$ involves a weakening of \emph{Standardisation}.  
In particular, the axiom \textsf{S} may be equivalently formulated as follows:
\be\tag{$\textsf{S}$}
(\forall^{\st}x)(\exists^{\st}y)\Phi(x, y)\di \big(\exists^{\st}F\big)(\forall^{\st}x)\Phi(x,F(x)),
\ee
for any formula $\Phi$ in the language $\IST$.  In this way, \textsf{S} may be viewed as a `standard' version of the axiom of choice.    
In light of the possible non-constructive content of the latter (See Footnote \ref{koen}), it is no surprise that \textsf{S} has to be weakened.  
In particular, $\P$ includes the following version of \textsf{S}, called the \emph{Herbrandised Axiom of Choice}, which is defined as follows:
\be\label{HACINT2}\tag{$\HAC_{\INT}$}
(\forall^{\st}x^{\rho})(\exists^{\st}y^{\tau})\varphi(x, y)\di \big(\exists^{\st}G^{\rho\di \tau^{*}}\big)(\forall^{\st}x^{\rho})(\exists y^{\tau}\in G(x))\varphi(x,y),
\ee
where $\varphi$ is any internal function in the language of $\P$.  Note that $G$ does not output a witness for $y$, but a \emph{finite list} of potential witnesses to $y$.  
This is quite similar to \emph{Herbrand's theorem} (\cite{buss}*{I.2.5}), hence the name of $\HAC_{\INT}$.  

\medskip

Finally, we list two basic but important axioms of $\P$ from Definition \ref{debs} below.  These axioms are inspired by Nelson's claim about $\IST$ as follows:
\begin{quote}
\emph{Every specific object of conventional mathematics is a standard set}. It remains unchanged in the new theory $[\IST]$. (\cite{wownelly}*{p.\ 1166})
\end{quote}
Specific objects of the system $\P$ obviously include the constants  $0, 1, \times, +$, and anything built from those.  Thus, the system $\P$ includes the following two axioms; we refer to Definition \ref{debs} for the exact technical details.     
\begin{enumerate}  
\item All constants in the language of $\textsf{E-PA}^{\omega}$ are standard.
\item A standard functional applied to a standard input yields a standard output. 
\end{enumerate}
As a result, the system $\P$ proves that any term of $\textsf{E-PA}^{\omega}$ is standard.  
This will turn out to be essential in Section \ref{RF}.  

\subsubsection{The constructive system $\H$}\label{graf2}
We discuss the fragment $\H$ of $\IST$ from \cite{brie}.   Similar to the way $\IST$ is an extension of $\ZFC$,
$\H$ is just the internal system $\textsf{E-HA}^{\omega}$ with the language extended with a new standardness predicate `\st' and with some special cases of the external axioms of $\IST$.  The technical details of this extension may be found in Section \ref{FULL}; we now provide an intuitive motivation for the external axioms of $\H$, assuming basic familiarity with the finite type system of G\"odel's system $T$ (also discussed in Section \ref{FULL}).  Note that $\textsf{E-HA}^{\omega}$ is based on \emph{intuitionistic logic} as introduced in Section \ref{hitch}.    

\medskip

First of all, the system $\H$ does not involve \emph{Transfer} for the same reasons $\P$ does not.  By contrast the axioms $\HAC_{\INT}$ and \eqref{wellhung} (and its contraposition) are included in $\H$, with the restrictions on $\varphi$ lifted even.  

\medskip

Secondly, the system $\H$ includes some `non-constructive' axioms relativized to `\st'.  We just mention the names of these axioms and refer to Section \ref{FULL} for a full description.  
The system $\H$ involves nonstandard versions of the following axioms: \emph{Markov's pricinciple} (See e.g.\ \cite{beeson1}*{p.\ 47}) and the independence of premises principle (See e.g.\ \cite{kohlenbach3}*{\S5}).  \emph{Nonetheless}, the system $\H$ proves the same internal sentence as $\textsf{E-HA}^{\omega}$ by Theorem \ref{TERM}, i.e.\ the nonstandard versions are not really non-constructive.  

\medskip 

Finally, $\H$ also includes the basic axioms from Definition \ref{debs} as listed at the end of Section \ref{graf1} above.

\subsubsection{The importance of $\P$ and $\H$}\label{fraki}
We discuss why $\H$ and $\P$ are important to our enterprise.  In a nutshell, these systems allow one to obtain computational content from Nonstandard Analysis by the following\footnote{The full version of Theorem \ref{TERM} is Corollary \ref{consresultcor} in Section \ref{FULL}.} `term extraction' theorem.  The scope of this theorem includes a huge part of Nonstandard Analysis as discussed in Section \ref{scope}.  
\begin{thm}[Term extraction]\label{TERM}
Let $\varphi$ be internal, i.e.\ not involving `$\st$'. \\
If $\P$ \(resp.\ $\H$\) proves $(\forall^{\st}x)(\exists^{\st}y)\varphi(x,y)$, then we can extract a term $t$ from this proof such that $\textsf{\textup{E-PA}}^{\omega}$ \(resp.\ $\textsf{\textup{E-HA}}^{\omega}$\) proves $(\forall x)(\exists y\in t(x))\varphi(x,y)$.  
The term $t$ is such that $t(x)$ is a finite list.    
\end{thm}
Note that the \emph{conclusion} of the theorem, namely `$\textsf{E-PA}^{\omega}$ proves $(\forall x)(\exists y\in t(x))\varphi(x,y)$', \emph{does not involve Nonstandard Analysis}. 
The term $t$ from the previous theorem is essentially a computer program formulated in e.g.\ Martin-L\"of type theory or Agda (\cites{loefafsteken, agda}).  
As such, $t$ is definitely an `algorithm' in the sense of Bishop (See Section \ref{hitch}).  

\medskip

As a first example of the ubiquity of computational content in Nonstandard Analysis, we now consider an elementary application of Theorem \ref{TERM} to nonstandard continuity as in Example \ref{leffe}.
More examples may be found in Section \ref{RF} and \cite{sambon}. 
\begin{exa}[Nonstandard and constructive continuity]\label{krel}\rm
Suppose $f$ is a function defined on the reals which is nonstandard continuous, provable in $\P$.  In other words, similar to Example \ref{leffe}, the following is provable in $\P$:
\be\label{frik}
(\forall^{\st}x\in \R)(\forall y\in \R)(x\approx y \di f(x)\approx f(y)).
\ee    
Since $\P$ includes \emph{Idealisation} (essentially) as in $\IST$, $\P$ also proves the following:
\begin{align}\label{exagoe}\textstyle
(\forall^{\st}x\in \R)(\forall^{\st}k\in \N)&(\exists^{\st}N\in\N)\\
&\textstyle\underline{(\forall y\in \R)\big( |x- y|<\frac{1}{N} \di|f(x)- f(y)|<\frac{1}{k}\big)},\notag
\end{align}
in exactly the same way as proved in Example \ref{leffe}.  Since the underlined formula in \eqref{exagoe} is internal, we note that Theorem \ref{TERM} applies to `$\P\vdash \eqref{exagoe}$'.  Applying the latter theorem, 
we obtain a term $t^{(1\times 0)\di 0^{*}}$ such that $\textsf{E-PA}^{\omega}$ proves:
\[\textstyle
(\forall  x\in\R, k\in \N)(\exists N\in t(x, k)){(\forall y\in \R)\big( |x- y|<\frac{1}{N} \di|f(x)- f(y)|<\frac{1}{k}\big)}.
\]
Since $t(x, k)$ is a finite list of natural numbers, define $s(x, k)$ as the maximum of $t(x, k)(i)$ for $i<|t(x, k)|$ where $|t(x,k)|$ is the length of the finite list $t(x,k)$.  
The term $s$ is called a \emph{modulus} of continuity of $f$ as it satisfies:
\[\textstyle
(\forall  x\in\R, k\in \N){(\forall y\in \R)\big( |x- y|<\frac{1}{s(x, k)} \di|f(x)- f(y)|<\frac{1}{k}\big)}.
\]
Similarly, from the proof in $\P$ that $f$ is nonstandard \emph{uniformly} continuous, we may extract a modulus of \emph{uniform} continuity (See Section \ref{CTT}).  
This observation is \emph{important}: moduli are an essential part of Bishop's \emph{Constructive Analysis} from Section~\ref{hitch} (See e.g.\ \cite{bish1}*{p.\ 34}), so we just proved that such constructive information \emph{is implicit in the nonstandard notion of continuity}!  
\end{exa}
By the previous example, the nonstandard notion of (uniform) continuity contains non-trivial \emph{constructive} information.  
It is a natural question how far this goes, i.e.\ how large is the scope of Theorem \ref{TERM}?  As it turns out, the scope of the latter is huge, as we discuss in Section \ref{scope} below.    
In a nutshell, other nonstandard definitions (of integration, differentiability, compactness, convergence, et cetera) behave in exactly the same way as continuity in Example~\ref{krel}, and the same holds for theorems 
solely formulated with these nonstandard definitions.

\subsection{Constructive Nonstandard Analysis}\label{palmke}
While most of this paper deals with \emph{classical} Nonstandard Analysis, we now discuss \emph{constructive} Nonstandard Analysis as its insights will be needed below.  
Note that constructive mathematics (in the sense of Bishop) is introduced in Section \ref{hitch}.  
We already have the system $\H$ as an example of the syntactic approach to constructive Nonstandard Analysis, and we now discuss the \emph{semantic approach}.  

\medskip

Intuitively, the \emph{semantic} approach to Nonstandard Analysis pioneered by Robinson (\cite{robinson1}) consists in somehow building a nonstandard model of a given structure (say the set of real numbers $\R$) and proving that the original structure is a strict subset of the nonstandard model (usually called the set of \emph{hyperreal numbers} $^{*}\R$) while establishing properties similar to \emph{Transfer}, \emph{Idealisation} and \emph{Standardisation} as theorems of this model and the original structure.  Historically, Nelson of course studied Robinson's work and axiomatised the semantic approach in his internal set theory $\IST$.  The most common way of building a suitable nonstandard model is using a \emph{free ultrafilter} (See e.g.\ \cite{loeb1,nsawork2}).  The existence of the latter is a rather strong \emph{non-constructive} assumption.            

\medskip

As it turns out, building nonstandard models with nice properties like \emph{Transfer} can also be done constructively:
Palmgren in \cite{opalm}*{Section 2} and \cite{nostpalm} constructs a
nonstandard model $\M$ (also called a `sheaf' model) satisfying the \emph{Extended Transfer Principle} by \cite{opalm}*{Corollary 4 and Theorem 5}.   
As noted by Palmgren (\cite{opalm}*{p.\ 235}), the construction of $\M$ can be formalised in 
Martin-L\"of's constructive type theory (\cite{loefafsteken}).  The latter was developed independently of Bishop's constructive mathematics (See Section \ref{hitch}), but can be viewed 
as a foundation of the latter.

\section{The Bishop-Connes critique}\label{PO}
We discuss the critique of Nonstandard Analysis by Errett Bishop (Section \ref{loper}) and Alain Connes (Section \ref{koko}).   This critique can be summarised as follows:  
\begin{center}
\emph{The presence of ideal objects \(in particular infinitesimals\) in Nonstandard Analysis yields the absence of computational content.}
\end{center}
and will be called the \emph{Bishop-Connes critique} (of Nonstandard Analysis).  
As may be expected in light of their different philosophies of mathematics\footnote{See Section \ref{hitch} (resp.\ Section \ref{koko}) for a discussion of Bishop's (resp.\ Connes') philosophy of mathematics.}, Connes and Bishop have diverging opinions on what constitutes `computational content', while they also have different areas in mind:  Bishop wrote on the foundations of mathematics while Connes had applications to physics in mind.    
Furthermore, the arguments for this critique by Bishop and Connes have been studied and found wanting, as discussed below.  
Nonetheless, this critique was never fully refuted, and we undertake this task in Section \ref{RF}.  

\subsection{Bishop's critique of Nonstandard Analysis}\label{loper}
Before we can discuss Bishop's critique of Nonstandard Analysis, we need to study his philosophy of mathematics in Section \ref{hitch}.  
\subsubsection{Bishop's philosophy of mathematics}\label{hitch}
There can be little doubt about the philosophical position of Bishop in light of his monograph \emph{Foundations of Constructive Analysis} (\cite{bish1}), in which the first chapter is titled \emph{A constructivist manifesto} and the preface reads:
\begin{quote}
This book is a piece of constructivist propaganda, designed to show
that there does exist a satisfactory alternative [to classical mathematics]. To this end we develop
a large portion of abstract analysis within a constructive framework. (\cite{bish1}*{p.\ ix})
\end{quote}
\begin{quote}
Our program is simple: To give numerical meaning to as much as possible of classical abstract analysis. Our motivation is the well-known scandal, exposed by Brouwer (and others) in great detail, that classical mathematics is deficient in numerical meaning. (\cite{bish1}*{p.\ ix})
\end{quote}
Bishop thus subscribes to \emph{constructivism}, a position which distinguishes itself from mainstream\footnote{Note that outside of the context of $\IST$, `mainstream' mathematics just has its usual meaning in this paper.} (or `classical') mathematics by the insistence that a mathematical object only exists once it has been constructed (in some way). 

\medskip

For instance, constructivists generally reject \emph{proof by contradiction} as this proof technique concludes\footnote{To prove that $(\exists x)A(x)$ by contradiction in classical mathematics, one assumes $(\forall x)\neg A(x)$ and derives a contradiction, i.e.\ one shows that $\neg[(\forall x)\neg A(x)]$.  Using the \emph{law of excluded middle $B\vee \neg B$}, one then concludes that $(\exists x)A(x)$} the existence of an object \emph{without constructing it}.  More generally, the law of excluded middle $B\vee \neg B$ (LEM for short) and axioms implying it\footnote{The axiom of choice implies LEM under certain conditions (\cite{dias}), but fragments of the axiom of choice are considered acceptable in constructive mathematics (See \cite{bridges1}*{\S1.4}).\label{koen}} are rejected in constructivism due to the lack of constructive content of LEM.     
This rejection becomes more palatable upon observing the interpretation of the logical symbols in constructivism, referred to as the \emph{Brouwer-Heyting-Kolmogorov} (BHK for short) interpretation (See \cite{bridge1}*{\S1}).  
\begin{defi}[BHK-interpretation]\label{kafi}
~\rm
\begin{enumerate}
\item The disjunction $P\vee Q$: we have an algorithm that outputs either $P$ or $Q$, together with a proof of the chosen disjunct.  
\item The conjunction $P\wedge Q$: we have both a proof of $P$ and a proof of $Q$.
\item The implication $P \di Q$: by means of an algorithm we can convert any proof of $P$ into a proof of $Q$.\label{hollyschijt}
\item The negation $\neg P$: assuming $P$, we can derive a contradiction (such as $0=1$); equivalently, we can prove $P\di (0=1)$.
\item The formula $(\exists x)P(x)$: we have (i) an algorithm that computes a certain object $x$, and (ii) an algorithm that, using the information supplied by the application of algorithm (i), demonstrates that $P(x)$ holds.
\item The formula $(\forall x\in A)P(x)$: we have an algorithm that, applied to an object $x$ and a proof that $x\in A$, demonstrates that $P(x)$ holds.
\end{enumerate}
\edefi
In light of the first item, the law of excluded middle LEM states the existence of an algorithm which can decide whether a given mathematical theorem is provable or not.  
Since nobody believes such an algorithm will ever be found, the rejection of LEM \emph{given the BHK interpretation} becomes clear.

\medskip

What is interesting about Bishop's version of constructive mathematics, is its `neutral' position, which we discuss in more detail.
Now, there are a number of approaches to constructive mathematics, as discussed at length in e.g.\ \cite{beeson1}*{III}, \cite{troeleke1}*{I.4}, or \cite{brich}, and Bishop's variety, called `BISH', can be said to occupy a `neutral' position between some of these and classical mathematics.  
In particular, any theorem of BISH is also a theorem of classical mathematics, a theorem of Brouwer's intuitionistic mathematics\footnote{Brouwer's intuitionism is discussed at length by Dummett in \cite{dummy}, including \emph{Brouwer's theorem} that all total functions on the unit interval must be (uniformly) continuous (\cite{dummy}*{Theorem 3.19}), which contradicts classical mathematics.}, and a theorem in the Russian school of recursive\footnote{The name `recursive' mathematics betrays that all mathematical objects must be recursive (called `computable' nowadays), as captured in the axiom \emph{Church's Thesis} (See \cite{brich}*{Chapter 3} or \cite{beeson1}*{I.8} for details), which contradicts classical mathematics.} mathematics, while the latter two are inconsistent with one another and with classical mathematics.   

\medskip

This neutral position comes at a price though:  To guarantee the aforementioned compatibility of BISH with classical, intuitionistic, and Russian recursive mathematics, the notion of `algorithm' is left unspecified by Bishop, \emph{despite its central role in light of the BHK-interpretation}.  
The following quote by Bridges captures the previous nicely.
\begin{quote}
Although Bishop has been criticised for being too vague in his concept of algorithm, by this very vagueness he left open the possibility of interpreting his work within a variety of formal systems. Not only is every theorem of BISH also a theorem of recursive constructive mathematics - which is, roughly, recursive function theory developed with intuitionistic logic - but it is also a theorem of Brouwer's intuitionistic mathematics, and, perhaps more significantly, of classical mathematics. (\cite{bridgetoofar}*{p.\ 2})
\end{quote}
Note that in this paper we do not judge this design choice made by Bishop: We merely point out a central aspect of BISH, namely a `know-it-when-you-see-it' approach to the notion of algorithm resulting in a `informal but rigorous' style.  

\medskip

Finally, we wish to point out that while Bishop was critical of classical mathematics, the \emph{main point} of his constructivist enterprise was \emph{not} this criticism, but to build a `computationally rich' alternative to classical mathematics.  
We consider two telling examples.  

\medskip

First of all, according to Bishop, while classical (or `idealistic') mathematics lacks computational content, it is not necessarily `worthless': 
\begin{quote}
\dots idealistic mathematics is [not] worthless from
the constructive point of view. This would be as silly as contending
that unrigorous mathematics is worthless from the classical point of
view. Every theorem proved with idealistic methods presents a 
challenge: to find a constructive version, and to give it a constructive proof. (\cite{bish1}*{p.\ x})
\end{quote}
Note that we do not claim that the `constructivisation' of classical mathematics is the only goal of Bishop or his followers (See \cite{nukino}*{p.\ 54} for a discussion).  

\medskip

Secondly, while Bishop is clear about his nominalist reservations regarding classical mathematics (going as far as comparing the latter to `God's mathematics' in \cite{bish1}*{\S1}), he does not have skeptical doubts about e.g.\ basic arithmetic (based on intuitionistic logic obviously).  For instance, he takes for granted the set of natural numbers and the associated axiom of induction. 
\begin{quote}   
The positive integers and their arithmetic are presupposed by
the very nature of our intelligence and, we are tempted to believe, by
the very nature of intelligence in general. The development of the
theory of the positive integers from the primitive concept of the unit,
the concept of adjoining a unit, and the process of mathematical
induction carries complete conviction. (\cite{bish1}*{\S1.1})
\end{quote}
Furthermore, Bishop has no problems accepting much more complicated mathematical objects (as is clear from the below quote), \emph{as long as} these have been built using (and can in principle be reduced to) algorithmic reasoning.  
\begin{quote}
Building on the positive integers, weaving a web of ever more sets
and more functions, we get the basic structures of mathematics: the
rational number system, the real number system, the euclidean spaces,
the complex number system, the algebraic number fields, Hilbert
space, the classical groups, and so forth. Within the framework of
these structures most mathematics is done. Everything attaches itself
to number, and every mathematical statement ultimately expresses the
fact that if we perform certain computations within the set of positive
integers, we shall get certain results. (\cite{bish1}*{\S1.1})
\end{quote}
We summarise that Bishop judged classical mathematics to be deficient in computational content, and that he took it upon himself to develop a kind of mathematics which is compatible with both classical and constructive approaches to mathematics and in which computational content is central.  In a nutshell, Bishop equates meaning with computational content, i.e.\ meaningful mathematics with mathematics as developed in BISH where every statement has numerical meaning and every object has an algorithmic description.  Finally, we are by no means the first to make this observation: the assertion that Bishop equates meaning and computational content is discussed in detail in \cite{kaka}*{\S6.2}.      

\subsubsection{An alternative view of constructive mathematics}\label{kikop}
 We described Bishop's philosophy of mathematics in the previous section.  We now present an alternative view of constructive mathematics first formulated by Richman (\cite{poorguy,poorguy2}).  The basic question we are considering is as follows:  
\begin{center}
What is the nature of objects in Bishop's Constructive Analysis?      
\end{center}
As is clear from the previous section, Bishop describes his mathematics as dealing with objects which are `given by an algorithm', as explicitly stated in e.g.\ \cite{bridges1}*{p.\ 14} and \cite{schizo}*{p.\ 15}, but the notion of algorithm is not defined.  We also saw that Bishop had good motivations for leaving the notion of algorithm undefined.  
We now discuss another view of constructive mathematics which sidesteps the previous question by placing the underlying (intuitionistic) logic at the forefront, while the constructive ontology `everything is given by algorithms' is de-emphasised.  

\medskip

In a nutshell, experience bears out that Bishop's Constructive Analysis simply amounts to \emph{mathematics using intuitionistic logic} in practice, i.e.\ the above question and the nature of Bishop's notion of algorithm can be sidestepped.  Bridges and Palmgren nicely formulate this as follows.     
\begin{quote}   
However, this criticism [that Bishop left the notion of algorithm undefined] can be overcome by looking more closely at what practitioners of BISH actually do, as distinct from what Bishop may have thought he was doing, when they prove theorems: in practice, they are doing mathematics with intuitionistic logic. Experience shows that the restriction to intuitionistic logic always forces mathematicians to work in a manner that, at least informally, can be described as algorithmic; so algorithmic mathematics appears to be equivalent to mathematics that uses only intuitionistic logic. If that is the case, then we can practice constructive mathematics using intuitionistic logic on any reasonably defined mathematical objects, not just some class of ``constructive objects''. 
(\cite{pabi}*{\S3.3}; quotes in the original)
\end{quote}
Obviously, the above emphasis on logic is at odds with the primacy of mathematics over logic that was part of the philosophy of Brouwer, Heyting, Markov, Bishop, and other pioneers of constructivism. On the other hand, as stated by Bridges in \cite{pabi}*{\S3.3}, this emphasis does capture the essence of constructive mathematics \emph{in practice}.

\subsubsection{Bishop on Nonstandard Analysis}
Bishop's constructivist convictions have been made clear, as well as the intentions of his program for the redevelopment of mathematics based on computational content.  Despite this predominant `positive' aspect of Bishop's enterprise, classical mathematics often received harsh criticism, and he even went as far as announcing its demise (which has not materialised so far) as follows.  
\begin{quote}
Very possibly classical mathematics will cease to exist as an independent discipline. (\cite{nukino}*{p.\ 54})
\end{quote}
Of course, Robinsonian\footnote{There are a number of constructive approaches to Nonstandard Analysis (See e.g.\ \cites{palmdijk,palm4,palm1,palm2,palm3, evenbellen}) which we briefly discuss in Section \ref{palmke}.} Nonstandard Analysis is part of classical mathematics, and therefore on the receiving end of Bishop's criticism of classical mathematics.  
Nonetheless, Bishop felt the need to single out Nonstandard Analysis on a number of occasions.  
We now consider three such negative statements by Bishop about Nonstandard Analysis, as they are relevant to the formulation of Bishop's view on Nonstandard Analysis.  A much more thorough discussion of these matters may be found in \cites{kaka}.  

\medskip

First of all, in the following quote Bishop criticises the apparent lack of meaning, which to him means `computational content', of Nonstandard Analysis, as well as the latter's introduction at the undergraduate level.  
\begin{quote}
A more recent attempt at mathematics by formal finesse is non-standard analysis. I gather that it has met with some degree of success, whether at the expense of giving significantly less meaningful proofs I do not know. My interest in non-standard analysis is that attempts are being made to introduce it into calculus courses. It is difficult to believe that debasement of meaning could be carried so far. \cite{kluut}*{p. 513}  
\end{quote}
Secondly, the following quote on the `meaning' of Nonstandard Analysis may be found in Bishop's notes from a summer school in New Mexico (\cite{bishl}).
\begin{quote}
[Constructive and Nonstandard Analysis] are at opposite poles.
Constructivism is an attempt to deepen the meaning of mathematics; non-standard analysis, an attempt to dilute it further. (\cite{bishl}*{p.\ 1-2})
\end{quote}
Thirdly, Bishop was asked to review Keisler's introduction to Nonstandard Analysis \cite{keisler3} which has been used for teaching at the undergraduate level.  
The final sentence of Bishop's review reads as follows.  
\begin{quote}
Now we have a calculus text that can be used to confirm their experience of mathematics as an esoteric and meaningless exercise in technique. \cite{bitch}*{p.\ 208}
\end{quote}
It is important to note that Bishop did not make his constructivist convictions explicit in \cite{bitch}, i.e.\ the reader is not informed that Bishop equates `meaning' and `computational content'.  
It should also be noted that Bishop's views are not necessarily shared by other constructivists.  For instance, the intuitionist Heyting spoke highly of Robinson's Nonstandard Analysis (\cites{heyting, kaka}).  

\medskip

As it happens, Bishop's review of \cite{keisler3} has been studied by historians of mathematics and classified rather unfavourably as follows:  Artigue \cite{art}*{p. 172} describes Bishop's review as `virulent'; Dauben (\cite{daupje}) as `vitriolic'; Davis and Hauser (\cite{dahaus}) as `hostile'; and Tall (\cite{tallmann}), as `extreme'.  Furthermore, Robinson himself added the following short but forceful judgement to his review of \cite{bish1}.  
\begin{quote}
The sections of [Bishop's] book that attempt to describe the philosophical and historical background of [the] remarkable endeavor [of Intuitionism] are more vigorous than accurate and tend to belittle or ignore the efforts of others who have worked in the same general direction (\cite{robninson2}*{p.\ 921}).      
\end{quote}
In light of these rather strong words, one has to wonder what specifically could be wrong with Nonstandard Analysis, among the plethora fields of classical mathematics, that prompted Bishop's ire and scorn? Katz and Katz (\cite{kaka}*{\S3.3}) discuss this question in detail and suggest an explanation based on three technical reasons, and a philosophical one.  
While their four arguments are perfectly valid in our opinion, we would rather `address the elephant in the room' as follows:  As noted in Sections \ref{hitch} and~\ref{kikop}, Bishop considers his mathematics to be concerned with \emph{constructive objects}, i.e.\ those described by algorithms on the integers.  The following quote is worth repeating.  
\begin{quote}  
Everything attaches itself
to number, and every mathematical statement ultimately expresses the
fact that if we perform certain computations within the set of positive
integers, we shall get certain results. (\cite{bish1}*{\S1.1})
\end{quote}
This quote suggest a \emph{fundamental} ontological divide between Bishop's mathematics and Nonstandard Analysis, as the latter \emph{by design} is based on ideal objects (like infinitesimals) with prima facia \emph{no algorithmic description at all}.   
In this light, it is not a stretch of the imagination to classify Nonstandard Analysis as \emph{fundamentally non-constructive}, i.e.\ antipodal to Bishop's mathematics, or in Bishop's words: \emph{an attempt to dilute meaning/computational content further}.  
Furthermore, this  `first impression' is only confirmed by the fact that the `usual' development of Nonstandard Analysis involves \emph{quite non-constructive}\footnote{The usual development of Robinson's Nonstandard Analysis proceeds via the construction of a nonstandard model using a free ultrafilter.  The existence of the latter is only slightly weaker than the axiom of choice of $\ZFC$ (\cite{nsawork2}).} axioms.  The latter is formulated by Katz and Katz as:
\begin{quote}
[\dots] the hyperreal approach incorporates an element of non-constructivity at the basic level of the very number system itself. (\cite{kaka}*{\S3.3})
\end{quote}
Finally, the aforementioned `first impression' does not seem to disappear if one considers more basic mathematics, e.g.\ arithmetic rather than set theory.
Indeed, \emph{Tennenbaum's theorem} \cite[\S11.3]{kaye} `literally' states that any nonstandard model of Peano Arithmetic is not computable.  \emph{What is meant} is that for a nonstandard model $\M$ of Peano Arithmetic, the operations $+_{\M}$ and $\times_{\M}$ cannot be computably defined in terms of the operations $+_{\N}$ and $\times_{\N}$ of the standard model $\N$ of Peano Arithmetic.  
In other words, Robinson's model-theoretic approach to Nonstandard Analysis seems fundamentally non-constructive even at the level of arithmetic.   

\medskip

In light of the above, we arrive at a possible explanation why Bishop singled out Nonstandard Analysis among all of classical mathematics:  Bishop's mathematics (in his view) deals with `constructive objects given by an algorithm' while (again in his view) Nonstandard Analysis by contrast is based on `non-constructive objects devoid of algorithmic description'.  Thus, while classical mathematics is obviously non-constructive in the sense of Bishop, Nonstandard Analysis seems to cheerfully take non-construcitivity to a whole new level by adopting non-constructive objects at a fundamental level.  Almost as an aside, we arrive at Bishop's critique of Nonstandard Analysis:
\begin{center}
\emph{The presence of ideal objects in Nonstandard Analysis yields the absence of computational content.}
\end{center}
From Bishop's point of view, the previous statement is almost trivial:  Constructive mathematics is built up from the ground (the integers) and all objects are based on algorithms.  
Introducing ideal objects, as is done \emph{at a fundamental level} in Nonstandard Analysis, is so contrary to the approach in Bishop's mathematics, the former must lead to a realm devoid of computational content.  To put it bluntly, Nonstandard Analysis thus occupies a special place in hell, i.e.\ classical mathematics, according to Bishop.    

\medskip

In conclusion, we have provided an explanation why Bishop singled out Nonstandard Analysis among all of classical mathematics.   
This explanation naturally led us to Bishop's critique of Nonstandard Analysis as formulated above.  Note that this critique is at least partially 
unfounded in light of the existence of \emph{constructive Nonstandard Analysis} as discussed in Section \ref{palmke}.  Nonetheless, this existence does not prove or disprove 
that \emph{classical} Nonstandard Analysis is fundamentally non-constructive.  

\subsection{Connes' critique of Nonstandard Analysis}\label{koko}
The Fields medallist Alain Connes has formulated negative criticism of classical Nonstandard Analysis in print on at least seven occasions.  
The first table in \cite{kano2}*{\S 3.1} runs a tally for the period 1995-2007. 
Connes judgements range from \emph{inadequate} and \emph{disappointing}, to \emph{a chimera} and \emph{irremediable defect}.  However, the judgements most interesting to us, especially in light of Bishop's critique from the previous section, pertain to the constructive\footnote{Note that Connes uses the word `constructive' as synonymous with `effective' and `explicit' from mainstream mathematics, i.e.\ no connection with constructivism seems present.\label{frikl}}/computable nature of infinitesimals and Nonstandard Analysis, as captured by the following quotes.  
\begin{quote}
Thus a non-standard number gives us canonically a non-measurable subset of $[0,1]$. This is the end of the rope for being `explicit' since (from another side of logics) one knows that it is just impossible to construct explicitely a non-measurable subset of $[0,1]$! 
(Verbatim copy of the text in \cite{conman3})
\end{quote}
\begin{quote}
 The point is that as soon as you have a non-standard number, you get a non-measurable set. And in Choquet's circle, having well studied the Polish school, we knew that every set you can name is measurable. So it seemed utterly doomed to failure to try to use non-standard analysis to do physics.  (\cite{conman}*{p.\ 26})
\end{quote}
\begin{quote}
The answer given by non-standard analysis, namely a nonstandard real, is equally disappointing: every non-standard real canonically determines a (Lebesgue) non-measurable subset of the interval $[0,1]$, so that it is impossible \cite{sterno} to exhibit a single [nonstandard real number]. The formalism that we propose will give a substantial and computable answer to this question. (\cite{conman4}*{p.\ 320})
\end{quote}
\begin{quote}
Suppose that a dart is thrown to the target of [Fig.\ 5 on \cite{conman2}*{p.\ 6207}, at the top]; then what is the probability of hitting a given point.
[In \cite{uncleberny} the authors claim that] the sought for [probability] makes sense, as a nonstandard positive real.  The problem with this proposed solution is that there is no way one can exhibit this infinitesimal.  [\dots] Our theory of infinitesimal variables is completely different [from Nonstandard Analysis], and it will give a precise computable answer to the above question
\end{quote}
Each of the previous quotes (and especially the last one) suggests that Nonstandard Analysis is somehow fundamentally non-constructive$^{\ref{frikl}}$ or non-computable in nature due to the presence of infinitesimals.  Thus, we may formulate Connes critique of Nonstandard Analysis as follows:
\begin{center}
\emph{The presence of ideal objects \(in particular infinitesimals\) in Nonstandard Analysis yields the absence of computational content.}
\end{center} 

As it happens, Connes refines this criticism in \cite{conman2} by formulating three major problems of Nonstandard Analysis which his proposed alternative formalism allegedly fixes.  
As discussed in \cite{kaka, kano2, gaanwekatten}, Connes' formalism suffers from at least two of these same problems.  In particular, Katz and Leichtnam note the following:  
\begin{quote}
Thus, two-thirds of Connes' critique of Robinson's infinitesimal approach can be said to be incoherent, in the specific sense of not being coherent with what Connes writes (approvingly) about his own infinitesimal approach. (\cite{gaanwekatten}*{p.\ 640})
\end{quote}
Furthermore, Katz and Leichtman note that Connes' formalism is rather non-constructive in nature, going as far as requiring the \emph{Continuum Hypothesis} in places (See \cite{gaanwekatten}*{Remark 5.1} for details), a statement known to be not provable or refutable in the usual foundation of mathematics \emph{Zermelo-Fraenkel set theory with the axiom of choice}.  

\medskip

Now, we should point out that Connes has previously espoused a platonist philosophy of mathematics (See e.g.\ \cite{caco}), i.e.\ the motivations for his criticism of Nonstandard Analysis do not seem to stem from a constructivist view of mathematics.  As is clear from the above quote from \cite{conman}*{p.\ 26}, Connes rather has pragmatic motivations in mind, namely that (fundamentally) non-constructive mathematics is useless for physics.  Indeed, a major aspect of physics is the testing of hypotheses against experimental data, nowadays done mostly on computers.  But how can this testing be done if the mathematical formalism at hand is fundamentally non-constructive, as Connes claims Nonstandard Analysis to be due to the presence of infinitesimals?

\medskip

In conclusion, Connes' critique of Nonstandard Analysis is similar to Bishop's, but with a different motivation (applicability to physics) and interpretation (no indication of a computational model is given).  

\subsection{Conclusion: The Bishop-Connes critique}
We have studied the critique of Nonstandard Analysis by Bishop and Connes.   We observed that Connes deems Nonstandard Analysis devoid of computational content due to the latter being based on ideal objects (infinitesimals in particular) at a fundamental level.  This makes Nonstandard Analysis unsuitable for physics, as literally claimed by Connes (See Section \ref{koko}), but no real foundational claims are made.  By contrast, we argued that Bishop's strong rejection of Nonstandard Analysis stems from his foundational beliefs:  In Bishop's version of constructivism, all objects are ultimately based on algorithms, while Nonstandard Analysis is fundamentally based on ideal objects (the most famous example being infinitesimals) with no (obvious) algorithmic content whatsoever.  Thus, we arrive at the Bishop-Connes critique:  
\begin{center}
\emph{The presence of ideal objects \(in particular infinitesimals\) in Nonstandard Analysis yields the absence of computational content.}
\end{center}
which stems from very different observations and beliefs, and in which `computational content' has different interpretations for Connes and Bishop. 

\medskip

Finally, we have identified the \emph{Bishop-Connes critique} as the main objection to Nonstandard Analysis by Bishop and Connes.  
We now point out an alternative view put forward by the second referee: 
\begin{quote}
While it is true that many people believe the Bishop-Connes critique (quite possibly including Bishop and Connes), this critique is probably more accurately viewed as one leg supporting their real critique of Nonstandard Analysis.  Taking it down does weaken the position (and especially, I suspect, its persuasiveness to those not already persuaded), but it is unlikely that Bishop,
Connes, or others strongly critical of the meaningfulness of nonstandard analysis would be moved by this article.
\end{quote}
Though beyond the scope of this paper,  we invite and welcome discussion on this topic, especially what other legs the Bishop-Connes critique rests on.  

\section{Reverse formalism}\label{RF}
\subsection{Introduction}\label{hintro}
As established in the previous sections, the Bishop-Connes critique of Nonstandard Analysis is as follows:
\begin{center}
\emph{The presence of ideal objects \(in particular infinitesimals\) in Nonstandard Analysis yields the absence of computational content.}
\end{center}
 This section is dedicated to establishing the opposite claim, namely 
 \begin{center}
\emph{The presence of ideal objects \(in particular infinitesimals\) in Nonstandard Analysis yields the \emph{ubiquitous presence} of computational content.}
\end{center}
In particular, we shall observe that infinitesimals provide an \emph{elegant shorthand} for expressing computational content. 
To this end, we shall exhibit a direct translation between (proofs of) theorems of Nonstandard Analysis and (proofs of) theorems rich in computational content, similar to the `reversals' from {Reverse Mathematics} discussed in Section \ref{RM}.  
The latter program also plays an important role in gauging the scope of this translation, in particular the `Big Five' classification, in Section \ref{scope}.  

\medskip

We label the aforementioned translation as a case of `reverse formalism', as it bestows Bishop-Connes style meaning (i.e.\ computational content) onto Nonstandard Analysis, and is similar to the reversals in Reverse Mathematics.  
We shall start with a basic example of reverse formalism stemming from probability theory in Section \ref{PTT}.  We further study two examples from analysis in Section \ref{CTT}, and consider a full-fledged theorem in Section \ref{HTT}.  As hinted at above, our notion of reverse formalism involves a (algorithmic) translation of proofs.  We discuss the scope of this translation in Section \ref{scope}, while possible criticism of this approach is discussed in Section \ref{promi}.

\subsection{A basic example from probability theory}\label{PTT}
We study our first example of reverse formalism:  a basic `computational' result from probability theory will be obtained from a qualitative result in Nonstandard Analysis.  As it happens, the latter has been used to model qualitative 
phenomena (See Section \ref{dva} for a considerable but non-exhaustive list), i.e.\ the potential applications are much richer than the following example.     
\begin{exa}\label{heathern}\rm
Suppose $P$ is a fixed probability measure and suppose $E_{1}, E_{2}$ are fixed collections of events expressible by an internal formula of $\P$.  
We would like to express the observation that if an event $A\in E_{1}$ is highly unlikely, then any event $B\in E_{2}$ is also highly unlikely.  This expression can be formalised as:
\be\label{imme}
(\forall^{\st} A\in E_{1})(\forall^{\st}B\in E_{2})\big(P(A)\approx 0 \di P(B)\approx 0    \big),  
\ee  
where we assume that $A, B$ are \emph{standard} as the latter predicate can also be interpreted as `observable' (See \cite{iamrobert}*{\S3.5.15} and \cite{goldie}*{\S5.1}).
The following does not change much if we drop the `\st' in \eqref{imme} for $A, B$:  The only change is that the extracted terms $s, t$ below do not depend on these variables anymore.    

\medskip

Now suppose \eqref{imme} is provable in $\P$; we use Theorem \ref{TERM} to obtain computational information from the former.  
Firstly, replace `$\approx$' by its definition:
\be\label{imme2}\textstyle
(\forall^{\st} A\in E_{1},B\in E_{2})\big((\forall^{\st}N\in \N)(P(A)<\frac{1}{N}) \di (\forall^{\st}k) (|P(B)|<\frac{1}{k})    \big).
\ee  
Secondly, bringing outside all standard quantifiers, we obtain:
\be\label{imme3}\textstyle
(\forall^{\st} A\in E_{1},B\in E_{2},k)(\exists^{\st}N\in \N) \big((P(A)<\frac{1}{N}) \di (|P(B)|<\frac{1}{k})    \big),  
\ee  
which has the right form to apply Theorem \ref{TERM} (i.e.\ to the statement `$\P\vdash \eqref{imme3}$').  Hence, we obtain a term $t$ such that $\textsf{E-PA}^{\omega}$ proves
\be\label{imme4}\textstyle
(\forall A\in E_{1}, B\in E_{2}, k)(\exists N\in t(A, B, k)) \big((P(A)<\frac{1}{N}) \di (|P(B)|<\frac{1}{k})    \big).
\ee  
Recalling that $t$ yields a finite list of terms, define $s(A, B, k)$ as the maximum of all $t(A, B, k)(i)$ for $i<|t(A, B, k)|$.  
We finally obtain that $\textsf{E-PA}^{\omega}$ proves
\be\label{imme5}\textstyle
(\forall A\in E_{1}, B\in E_{2}, k) \big((P(A)<\frac{1}{s(A, B, k)}) \di (|P(B)|<\frac{1}{k})    \big),
\ee  
where $s$ is a computer program as discussed in Section~\ref{fraki}.  
\end{exa}
In the previous example, we started from \eqref{imme} which expresses that if some event $A\in E_{1}$ is unlikely, then so is every event $B\in E_{2}$, and we obtained \eqref{imme5} which expresses that to guarantee that $B\in E_{2}$ has probability less than $\frac{1}{k}$, $A\in E_{1}$ should have probability less than $\frac{1}{s(A, B, k)}$.  
Thus, we also know `how unlikely' $A$ should be to guarantee that $B$ has less than $\frac{1}{k}$ probability  \emph{thanks to the statement \eqref{imme}}. 

\medskip

In contrast to the Bishop-Connes critique, statements of Nonstandard Analysis (also involving infinitesimals) like \eqref{imme} \emph{do} have some computational content, namely \eqref{imme5}.    
As noted at the beginning of this section, \eqref{imme} is just one example among many (See also Example \ref{bayes} below).  

\medskip

Next, we show that \eqref{imme5}, the `computational version of \eqref{imme}', is also `meta-equivalent' to \eqref{imme} as follows.  
\begin{thm}\label{lacky} ~
\begin{enumerate}
\item From a proof in $\P$ of \eqref{imme}, a term\footnote{As noted in Section \ref{fraki}, terms obtained from Theorem \ref{TERM} are indeed algorithms, and we use these two words interchangeably in the context of system $\P$.} $s$ can be extracted such that \eqref{imme5} is provable in $\textsf{\textup{E-PA}}^{\omega}$.  
\item Let $s$ be a term.  From a proof in $\textsf{\textup{E-PA}}^{\omega}$ of \eqref{imme5}, a proof in $\P$ of \eqref{imme} can be extracted.  
\item Let $s$ be a term.  The system $\H$ proves that $\eqref{imme5}\di \eqref{imme}$.  
\end{enumerate}
\end{thm}
\begin{proof}
The first item has been established in Example \ref{heathern}.  For the second item, since $\P$ is an extension of $\textsf{E-PA}^{\omega}$, a proof of \eqref{imme5} in the latter is also valid in the former.   
Furthermore, as discussed in Section \ref{graf1}, the axioms of $\P$ imply (i) that every term in the language of \textsf{E-PA}$^{\omega}$ is standard and (ii) that a standard object applied to standard input yields standard output (See Definition \ref{debs}).  
In particular, for the term $s$ in \eqref{imme5}, we have that $s(k)$ is standard for standard $k, A, B$.  
Hence, `$\textsf{E-PA}^{\omega}\vdash \eqref{imme5}$' yields a proof in $\P$ of \eqref{imme3} (taking $N=s(k, A, B)$), and the latter formula immediately implies \eqref{imme}.   
For the third item, it suffices to observe that $\H$ also includes the basic axioms from Definition \ref{debs} and repeat the previous part of the proof.    
\end{proof}
The previous theorem is our first example of \emph{reverse formalism}:  We start from a statement \eqref{imme} in Nonstandard Analysis and obtain its `computational version' \eqref{imme5} thanks to Theorem \ref{TERM}.  Perhaps surprisingly, \eqref{imme5} in turn implies \eqref{imme} in the sense of the second and third item of Theorem \ref{lacky}.  In this way, we may call \eqref{imme} and \eqref{imme5} `meta-equivalent' similar to the reversals in Reverse Mathematics, as the former statements are `equivalent in the meta-theory'.  Furthermore, while these statements are equivalent in the aforementioned sense, \eqref{imme} is arguably a simpler statement than \eqref{imme5}, lending credence to our claim that \emph{infinitesimals provide an elegant shorthand for expressing computational content}. 

\medskip
     
Moreover, in the proof of the second item of Theorem \ref{lacky}, the term $s$ as in \eqref{imme5} is standard inside $\P$ and thus gives rise to \eqref{imme3}.  The latter in turn implies \eqref{imme}, i.e.\ computational information (in the form of $s$) is converted to \emph{external} predicates (`$\approx$' in \eqref{imme} and `$\st$' in \eqref{imme3}).  Similarly, in the proof of the first item, applying Theorem \ref{TERM} to \eqref{imme3} (which follows from \eqref{imme}), a term $t$ is obtained computing an upper bound on $N$ in terms of $k, A, B$.  In particular, the latter four variables occur as \emph{standard quantifiers} in \eqref{imme3} while the rest of the formula (which is internal) is ignored, i.e.\ no computational content is provided with regard to e.g.\ `$P(B)<\frac{1}{k}$'.           
In other words, external predicates (`$\approx$' in \eqref{imme} and `$\st$' in \eqref{imme3}) are translated to computational information (in the form of $t$ and $s$).  
In conclusion, external predicates in \eqref{imme} give rise to the computational content in \eqref{imme5} \emph{and vice versa}.  This observation nicely supports our claim from the introduction that 
 \begin{center}
\emph{The presence of ideal objects \(in particular infinitesimals\) in Nonstandard Analysis yields the \emph{presence} of computational content.}
\end{center}
Actually, the results in Theorem \ref{lacky} and previous discussion suggest that the presence of nonstandard objects is somehow \emph{equivalent} (in the meta-theory) to the presence of computational content.
We will discuss the \emph{ubiquity} (as claimed in Section \ref{hintro}) of said computational content in Section \ref{scope}.  

\medskip

Next, suggested by the third item of Theorem \ref{lacky}, there is also constructive content (in the sense of constructive mathematics) to be found in Theorem \ref{lacky}:
Most importantly, we observe that the step from the `computational version' \eqref{imme5} to the nonstandard version \eqref{imme} can be done \emph{in the constructive system $\H$} (and actually only the basic axioms from Definition \ref{debs} are needed).  
In turn, the translation from \cite{brie} used in Theorem~\ref{TERM} and the first item of Theorem \ref{lacky} can be formalised in any reasonable\footnote{Here, a `reasonable' system is one which can prove the usual properties of finite lists, for which the presence of the exponential function suffices.  In particular, a subsystem of \emph{primitive recursive arithmetic}, where the latter is claimed to correspond to Hilbert's finitist mathematics (\cite{tait1}), suffices.} system of constructive mathematics.  In fact, the formalisation of the results in \cite{brie} 
in the proof assistant Agda (based on Martin-L\"of's constructive type theory \cite{loefafsteken}) is underway in \cite{EXCESS}.  

\medskip

Finally, we list one more example of a qualitative statement from Nonstandard Analysis which the reader may study in the same way as Example \ref{heathern}.
\begin{exa}\label{bayes}\rm
An intuitive statement about conditional probabilities is:
\[
\textup{If $P(A)$ is low and $P(B)$ is not low, then $P(A|B)$ is low,} 
\]
and the latter may be formalised as follows in Nonstandard Analysis:
\be\label{EXA44}
(\forall A, B)\big((P(A)\approx 0 \wedge P(B)\not\approx 0)\di P(A|B)\approx 0\big).
\ee
Assuming\footnote{It is easy to prove \eqref{EXA44} in $\P$ using Bayes' theorem $P(A|B)=\frac{ P(B|A)P(A)}{P(B)}$ for $P(B)\ne 0$, and basic properties of infinitesimals.} that $\P$ proves \eqref{EXA44}, Theorem \ref{TERM} yields a term $s$ such that
\[\textstyle
(\forall k,k' \in \N)(\forall A, B)\big((P(A)<s(k, k') \wedge P(B)>\frac{1}{k})\di P(A|B)<\frac{1}{k'}\big),
\]
and a version of Theorem \ref{heathern} can be obtained easily.  Note that the previous formula expresses \emph{how unlikely} $A$ and how likely $B$ have to be to guarantee that $P(A|B)$ is unlikely in the sense of being less than $\frac{1}{k'}$.     
\end{exa}
\subsection{Basic examples involving continuity and convergence}\label{CTT}
We study our second example of reverse formalism:  basic `constructive' definitions from analysis will be obtained from the associated definitions in Nonstandard Analysis.
We study uniform continuity and convergence inspired by Example \ref{krel}.  
These notions are defined as follows in Nonstandard Analysis. 
\bdefi\label{Kont}
A function $f$ is \emph{nonstandard uniformly continuous} on $[0,1]$ if
\be\label{soareyou4}
(\forall x, y\in [0,1])[x\approx y \di f(x)\approx f(y)].
\ee
A sequence $x_{(\cdot)}$ \emph{nonstandard converges} to $x$ if 
\be\label{convent}
(\forall N\in \N)\big(\neg\st(N) \di x_{N}\approx x\big).  
\ee
\edefi
The `constructive' definitions of uniform continuity and convergence (as employed by Bishop) are just the `usual' epsilon-delta definitions with moduli.   
\bdefi\label{CKont}
$f$ is \emph{uniformly continuous} on $[0,1]$ with modulus $s$ if
\be\label{soareyou44}\textstyle
(\forall  k\in \N){(\forall x,y\in [0,1])\big( |x- y|<\frac{1}{s(k)} \di|f(x)- f(y)|<\frac{1}{k}\big)}.
\ee
A sequence $x_{(\cdot)}$ \emph{converges} to $x$ with modulus $N$ if 
\be\label{convent4}\textstyle
(\forall k,n\in \N)( n\geq N(k) \di |x_{n}-x|<\frac{1}{k}  )
\ee
\edefi
Similar to Theorem \ref{lacky} and Example \ref{krel}, we have the following theorem.  
\begin{thm}\label{lacky2} ~
\begin{enumerate}
\item From a proof in $\P$ of \eqref{soareyou4} \(resp.\ \eqref{convent}\), a term\footnote{As noted in Section \ref{fraki}, terms obtained from Theorem \ref{TERM} are indeed algorithms, and we use these two words interchangeably.} $s$ \(resp.\ $N$\) can be extracted such that \eqref{soareyou44} \(resp.\ \eqref{convent4}\) is provable in $\textsf{\textup{E-PA}}^{\omega}$.  
\item Let $s,N$ be terms.  From a proof in $\textsf{\textup{E-PA}}^{\omega}$ of \eqref{soareyou44} \(resp.\ \eqref{convent4}\), a proof in $\P$ of \eqref{soareyou4} \(resp.\ \eqref{convent}\) can be extracted.  
\item Let $s,N$ be terms.  The system $\H$ proves $\eqref{soareyou44}\di \eqref{soareyou4}$ and $\eqref{convent4}\di \eqref{convent}$.  
\end{enumerate}
\end{thm}
\begin{proof}
For the first item, the case of uniform continuity is proved in the same way as in Example \ref{krel}.  For the case of convergence, \eqref{convent} implies by definition:
\be\label{convent2}\textstyle
(\forall N\in \N)\big((\forall^{\st}m)(N\geq m) \di (\forall^{\st} k)( |x_{N}- x|<\frac{1}{k}\big).  
\ee
Bringing outside the standard quantifiers as far as possible, we obtain
\be\label{convent3}\textstyle
(\forall^{\st} k)\underline{(\forall N\in \N)(\exists^{\st}m)\big(N\geq m \di  |x_{N}- x|<\frac{1}{k}\big)}.  
\ee
Applying idealisation \textsf{I} to the underlined formula, we obtain a finite list $z$ of natural numbers such that $(\forall N\in \N)(\exists m\in z)$ in \eqref{convent3}.  With $m_{0}$ equal to the maximum of all elements in $z$, \eqref{convent3} becomes
\be\label{convent33}\textstyle
(\forall^{\st} k)(\exists^{\st}m_{0}){(\forall N\in \N)\big(N\geq m_{0} \di  |x_{N}- x|<\frac{1}{k}\big)}.  
\ee
Applying Theorem \ref{TERM} to `$\P\vdash \eqref{convent33}$', we obtain a term $t$ such that   
\be\label{convent333}\textstyle
(\forall k)(\exists m_{0}\in t(k)){(\forall N\in \N)\big(N\geq m_{0} \di  |x_{N}- x|<\frac{1}{k}\big)}
\ee
is provable in $\textsf{E-PA}^{\omega}$.  Define $N(k)$ as the maximum natural number in the list $t(k)$ and note that $N$ is as required by the theorem.    

\medskip

For the second item, a term $N$ is standard for standard input in $\P$, hence a proof of \eqref{convent4} in \textsf{E-PA}$^\omega$ yields a proof of \eqref{convent33} in $\P$.  
The latter immediately implies \eqref{convent}, as all nonstandard numbers are bigger than all standard ones.  The second item for the case of uniform continuity is similar.    The third item is proved in exactly the same way as the second one.  \qed
\end{proof}
The previous theorem is our second example of \emph{reverse formalism}, and the same observations as in the previous section can be made:  For the nonstandard statements \eqref{soareyou4} and \eqref{convent}, there are `meta-equivalent' computational versions \eqref{soareyou44} and \eqref{convent4}.  The nonstandard versions are much shorter, and their use of infinitesimals as in `$\approx$' provides an elegant shorthand for the existence of moduli.  In particular, the latter give rise to the predicates `$\approx$' in the nonstandard versions, \emph{and vice versa} by (the proof of) Theorem \ref{lacky2}.  As discussed in Section \ref{scope}, similar results exist for other basic notions from analysis.    

\medskip

The modulus functions from Theorem \ref{lacky2} are important for the following reason: moduli for (uniform) continuity and convergence are \emph{indispensable\footnote{Bishop uses the exact words `indispensable part' with regard to moduli in \cite{bish1}*{p.\ 34}.} parts} of the constructive definitions of these notions (See \cite{bish1}*{p.\ 34 and  p.\ 26}) in Bishop's constructive mathematics.  
Thus, in direct contradiction with Bishop's critique of Nonstandard Analysis, the nonstandard definitions of continuity and convergence have computational content, namely the 
same `indispensable' constructive content required by Bishop.  

\medskip

In conclusion, we have obtained a second example of reverse formalism which seems to generalise nicely to other basic notions from analysis, anticipating the huge scope of Theorem \ref{TERM} to be discussed in Section \ref{scope}.  
Our example even included constructive content in the form of Bishop's `indispensable' modulus functions.  

\subsection{A basic example from high-school mathematics}\label{HTT}
We discuss an example of reverse formalism based on an actual theorem of (in some parts of the world) high-school mathematics.  
\subsubsection{Preliminaries}
The most basic notions of analysis, going back to high-school mathematics, include \emph{continuity} and \emph{Riemann integration}.  The former was introduced in the previous sections, while the latter is defined as follows.     
\bdefi[Riemann integration]\label{kunko}~
\begin{enumerate}
\item A \emph{partition} of $[0,1]$ is any sequence $\pi=(0, t_{0}, x_{1},t_{1},  \dots,x_{M-1}, t_{M-1}, 1)$.  We write `$\pi \in P([0,1]) $' to denote that $\pi$ is such a partition.
\item For $\pi\in P([0,1])$, $\|\pi\|$ is the \emph{mesh}, i.e.\ the largest distance between two adjacent partition points $x_{i}$ and $x_{i+1}$. 
\item For $\pi\in P([0,1])$ and $f:\R\di \R$, the real $S_{\pi}(f):=\sum_{i=0}^{M-1}f(t_{i}) (x_{i}-x_{i+1}) $ is the \emph{Riemann sum} of $f$ and $\pi$.  
\item A function $f$ is \emph{nonstandard integrable} on $[0,1]$ if
\be\label{soareyou5}
(\forall \pi, \pi' \in P([0,1]))\big[\|\pi\|,\| \pi'\|\approx 0  \di S_{\pi}(f)\approx S_{\pi'}(f)  \big].
\ee
\end{enumerate}
\edefi
Bishop proves the following theorem regarding (uniform) continuity and Riemann integration inside BISH in \cite{bish1}*{Theorem 9}.
\begin{thm}[$\RIE_{\ef}(t)$]\label{kliit}
If a function is uniformly continuous on the unit interval, then it is Riemann integrable there, i.e.\ for all $f: \R\di \R$ and $g:\N\di \N$, we have
\begin{align}\textstyle
(\forall \textstyle x, y &\in [0,1],k\textstyle)(|x-y|<\frac{1}{g(k)} \di |f(x)-f(y)|\leq\frac{1}{k})\label{EST}\\
&\textstyle\di  (\forall n)(\forall \pi, \pi' \in P([0,1]))\big(\|\pi\|,\| \pi'\|< \frac{1}{t(g,n)}  \di |S_{\pi}(f)- S_{\pi'}(f)|\leq \frac{1}{n} \big) ,\notag
\end{align}
\end{thm}
Here, $t$ is a computer program or `algorithm' in the sense of Bishop's Constructive Analysis.  Recall that a `modulus of uniform continuity $g$' as in \eqref{EST} is part and parcel 
of constructive continuity (See \cite{bish1}*{Def.\ 9}).  Thus, Bishop's theorem provides an algorithm (namely $t$) to compute a `modulus of Riemann integration' $\frac{1}{t(g, \cdot)}$ from a modulus of uniform continuity $g$.  

\medskip

Now consider the nonstandard version of the theorem that uniform continuity implies Riemann integration on the unit interval.  
\begin{thm}[$\RIE_{\ns}$]\label{kleet}
Every function $f:\R\di \R$ is nonstandard integrable on $[0,1]$ if it is nonstandard uniformly continuous there.  
\end{thm}
Theorem \ref{kleet} is clearly part of Nonstandard Analysis, and even commits the `original sin' (according to Connes) of non-constructivity by mentioning infinitesimals\footnote{The predicate `$x\approx y$' is usually read as `the distance between $x$ and $y$ is infinitesimal'.}.  
Hence, $\RIE_{\ns}$ should be devoid of computational content according to the Biship-Connes critique.  
Nonetheless, we have the following theorem which was first proved in \cite{sambon}*{\S3.1}.   
\begin{thm}\label{kater}
From a proof of $\RIE_{\ns}$ in $\P$, an algorithm $t$ can be extracted such that $\RIE_{\ef}(t)$ is provable in $\textup{\textsf{E-PA}}^{\omega}$.   Each step in the proof of $\RIE_{\ns}$ is a step in the algorithm $t$.  The same holds for $\H$ and $\textsf{\textup{E-HA}}^{\omega}$.  
\end{thm}
\begin{proof}
The proof of \cite{loeb1}*{Prop.\ 12.3} goes through for nonstandard uniformly continuous functions as in $\RIE_{\ns}$ (without the use of \emph{Transfer}).   
Similar to the treatment of continuity in the previous section, $\RIE_{\ns}$ can be brought in the right form to apply Theorem \ref{TERM}.  The extracted term $t$ yields a term 
$s$ such that $\RIE_{\ef}(s)$.  
For details, we refer to the proof of Theorem 3.3 in \cite{sambon}*{\S3.1}.  
\end{proof}
We conclude that $\RIE_{\ns}$ is a theorem of classical Nonstandard Analysis (namely provable in $\P$), and that Theorem \ref{TERM} allows us to extract considerable computational content.  In particular, from the proof of $\RIE_{\ns}$, we can `read off' the algorithm $t$ such that $\RIE_{\ef}(t)$, which is a theorem of BISH.   

\medskip

Thus, the Theorem \ref{kater} at least partially contradicts the Bishop-Connes critique as we have exhibited the non-trivial computational content of a theorem of Nonstandard Analysis.  
The case of $\RIE_{\ns}$ is not an isolated accident:  As it turns out, the `term extraction theorem', namely Theorem \ref{TERM},  has an extremely wide scope as discussed in Section \ref{scope}.  

\medskip   

In the next section, we `upgrade' $\RIE_{\ef}(t)$ to a `more constructive version' which is meta-equivalent (as in Theorems \ref{lacky} and \ref{lacky2}) to $\RIE_{\ns}$.   
This provides us with a non-trivial example of reverse formalism.  

\subsubsection{Another example of reverse formalism}
We have observed that classical Nonstandard Analysis has computational content by way of an illustrative example involving $\RIE_{\ns}$ and $\RIE_{\ef}(t)$.  
We now show that the former is intimately linked (meta-equivalent) to an `even more constructive' version $\RIE_{\pw}(s)$.  In particular, we show that $\RIE_{\pw}(s)$ and $\RIE_{\ns}$ are one and the same theorem up to a syntactical translation of their proofs.  

\medskip

First of all, $\RIE_{\pw}(s)$ is defined as a `very constructive' version of $\RIE_{\ef}(t)$.  
\begin{thm}[$\RIE_{\pw}(s)$] For all $f:\R\di \R$, $g:\N\di \N$, and $k'$ we have
\begin{align*}\textstyle
(\forall k\leq s(g,k'))&\textstyle(\forall \textstyle x, y \in [0,1])(|x-y|<\frac{1}{g(k)} \di |f(x)-f(y)|\leq\frac{1}{k})\\
&\label{FEST}\textstyle\di  (\forall \pi, \pi' \in P([0,1]))\big(\|\pi\|,\| \pi'\|< \frac{1}{s(g,k')}  \di |S_{\pi}(f)- S_{\pi}(f)|\leq \frac{1}{k'} \big).
\end{align*}
\end{thm}
As suggested by the notation, $\RIE_{\pw}(s)$ is indeed a `pointwise' version of $\RIE_{\ef}(t)$:  While the latter applies to continuous functions, the former applies 
to functions which are only continuous up to some precision (as measured by `$(\forall k\leq s(g,k'))$' in the antecedent).  In particular, $\RIE_{\pw}(s)$ tells us `how much' continuity is needed to approximate the Riemann integral up to precision $\frac{1}{k'}$.    

\medskip

The following theorem tells us that $\RIE_{\ns}$ can be translated to $\RIE_{\pw}(s)$ \emph{and vice versa}.  
In this way, $\RIE_{\pw}(s)$ bestows its computational content onto $\RIE_{\ns}$, a theorem of Nonstandard Analysis, i.e.\ we have found our first example of `reverse formalism'.  

\begin{thm}\label{key}~
\begin{enumerate}
\item From a proof of $\RIE_{\ns}$ in $\P$, an algorithm $s$ can be extracted such that $\RIE_{\pw}(s)$ is provable in $\textup{\textsf{E-PA}}^{\omega}$.  
\item Let $s$ be an algorithm.  From a proof in $\textsf{\textup{E-PA}}^{\omega}$ of $\RIE_{\pw}(s)$, a proof in $\P$ of $\RIE_{\ns}$ can be extracted.  
\item Let $s$ be an algorithm; then $\H$ proves that $\RIE_{\pw}(s)\di \RIE_{\ns}$.   
\end{enumerate}
\end{thm}
\begin{proof}
The first item is proved in \cite{sambon}*{Rem.\ 3.9}; similar results may be found in \cite{samzooII}.  For the other items, it suffices to note that $v$ produces standard output for standard input in $\P$ and $\H$, which follows form the `basic axioms' discussed in Sections \ref{graf1} and \ref{graf2}.  In this way, $\RIE_{\ns}$ easily follows from $\RIE_{\pw}(v)$ in the same way as in the proofs of Theorems \ref{lacky} and \ref{lacky2}.    \qed
\end{proof}
It is interesting to note that the third item can be proved using more `elementary' systems than $\H$, i.e.\ we only need the basic axioms from Definition~\ref{debs}.  

\medskip

Again, the connection between $\RIE_{\ns}$ and $\RIE_{\pw}(v)$ is not an isolated accident:  it is shown in \cites{sambon, samzooII} that term extraction via Theorem \ref{TERM} can always produce `pointwise' theorems like $\RIE_{\pw}(v)$, which in turn imply the original theorem of Nonstandard Analysis as in Theorem \ref{key}.   

\medskip

The previous theorem is our third example of \emph{reverse formalism}, and the same observations as in the previous sections can be made:  For the nonstandard statement $\RIE_{\ns}$, there is a `meta-equivalent' computational version $\RIE_{\pw}(s)$.  The nonstandard version is much shorter, and its use of infinitesimals as in `$\approx$' provides an elegant shorthand for the existence of moduli.  In particular, the latter give rise to the predicates` $\approx$' in the former, \emph{and vice versa} by Theorem \ref{key}.  As discussed in Section \ref{scope}, similar results exist for a large part of theorems from Nonstandard Analysis.    

\medskip 

What is left is to discuss is the scope of reverse formalism (Section \ref{scope}) and possible criticism (Section \ref{promi}).

\subsection{The scope of reverse formalism}\label{scope}
We discuss the scope of reverse formalism and the central Theorem \ref{TERM} in particular.  In light of the latter, formulas of the form $(\forall^{\st}x)(\exists^{\st}y)\varphi(x,y)$ ($\varphi$ internal) play a central role and we shall refer to such formulas as `normal forms'.  Based on the latter, we introduce \emph{pure} Nonstandard Analysis, i.e.\ that part of the latter falling within the scope of Theorem \ref{TERM}, and hence reverse formalism.  This turns out to be a considerable part of Nonstandard Analysis, as is clear from the following \emph{informal} description.     
\begin{des}[Pure Nonstandard Analysis]\label{pure}  
A theorem of \emph{pure} Nonstandard Analysis is built up as follows.   
\begin{enumerate}
\renewcommand{\theenumi}{\roman{enumi}} 
\item Only \emph{nonstandard definitions} (of continuity, compactness, \dots) are used; \textbf{no} epsilon-delta definitions are used.\label{controlisbetter}   
The former have (nice) normal forms and give rise to the associated constructive definitions from Figure \ref{ananlo3}.
\item Normal forms are closed under \emph{implication} by Theorem \ref{nogwelconsenal}.  
\item Normal forms are closed under \emph{prefixing a quantifier over all infinitesimals} by Theorem \ref{hujiku}, i.e.\ if $\Phi(\eps)$ has a normal form, so does $(\forall \eps\approx 0)\Phi(\eps)$. 
\item Fragments of the axioms \emph{Transfer} and \emph{Standard Part} corresponding to the Big Five of Reverse Mathematics have normal forms, as shown in \cite{sambon}*{\S4}.\label{notrep}
\item Formulas\footnote{We stress that item \ref{horgi} should be interpreted in a specific narrow technical sense (beyond the scope of this paper), namely as discussed in \cite{sambon2}.} involving the \emph{Loeb measure} (See \cite{sambon2}) have normal forms.\label{horgi}
\end{enumerate}
\end{des}
Regarding item \ref{controlisbetter} and as noted in Example \ref{krel}, nonstandard definitions give rise to the associated constructive definition-with-a-modulus.  
The following list provides an overview for common notions\footnote{A space is \emph{F-compact} in Nonstandard Analysis if there is a discrete grid which approximates every point of the space up to infinitesimal error, i.e.\ the intuitive notion of compactness from physics and engineering.\label{fookie}}, where $\Paai$ is \emph{Transfer} limited to $\Pi_{1}^{1}$-formulas (See \cite{sambon}*{\S4} or Section \ref{expl}), while $(\mu^{2})$ and $(\mu_{1})$ are the functional versions of $\ACA_{0}$ and $\FIVE$ (See \cite{avi2} or Section \ref{expl}).      
\begin{figure}[h!]
\begin{center}
\begin{tabular}{|c|c|}
\hline
Nonstandard Analysis definition  & Constructive/functional definition \\
  \hline \hline
 nonstandard convergence & convergence with a modulus\\
   \hline  
 nonstandard (uniform) continuity & (uniform) continuity with a modulus\\
\hline
  F-compactness$^{\ref{fookie}}$ & total boundedness \\
\hline
  nonstandard differentiability  & differentiability with a modulus  \\
   & and the derivative given \\
\hline
   nonstandard Riemann integration & Riemann integration with a modulus \\\hline\hline
      $\paai$ & Feferman's mu-operator $(\mu^{2})$ \\\hline
      $\Paai$ & Feferman's second mu-operator $(\mu_{1})$ \\\hline
\hline
\end{tabular}
\end{center}
\caption{Nonstandard and constructive definitions}
\label{ananlo3}
\end{figure}~\\
Regarding item \ref{notrep}, it is shown in \cite{sambon} that Theorem~\ref{TERM} applies to (nonstandard versions of) the Big Five of Reverse Mathematics; For completeness, three representative examples are given in Section~\ref{EXAKE} below. 
Furthermore, it is shown in \cite{samzoo, samzooII} that the same holds for the theorems in the Reverse Mathematics zoo (which gathers exceptions not fitting within the Big Five).  
By the observations in Section \ref{RM}, Theorem \ref{TERM} seems to applies to \emph{most of ordinary mathematics}.   

\medskip

In other words, Theorem \ref{TERM} allows us to extract computational content from theorems of Nonstandard Analysis, while the scope of the former theorem is \emph{most of ordinary mathematics}, as qualified by the classification of Reverse Mathematics in Section \ref{RM}.  The computational content of classical Nonstandard Analysis is thus seen to be vast.  

\medskip

The previous observation is important as we noted in Section \ref{hotgie} that Bishop stresses the importance of studying actual mathematics rather than -in his own words- dinky formal systems (but see Section \ref{informo}).  
For this reason, we list some of the theorems studied in \cite{sambon} in the same way as Theorem \ref{key}: The fundamental theorem of calculus, the Picard and Peano existence theorems, the intermediate value theorem, the uniform limit theorem, Dini's theorem, Heine's theorem, the monotone convergence theorem, the Heine-Borel lemma, the Stone-Weierstra\ss~theorem, and the Weierstra\ss~maximum and approximation theorems.  
Twice as many `nameless' theorems are studied in \cite{sambon}.  
Several representative examples are listed in Section \ref{EXAKE}. 

\section{Conclusion}
We formulate the conclusion to this paper and formulate possible criticism of this conclusion, reverse formalism in particular.   
\subsection{Conclusion on reverse formalism}
We have provided examples of reverse formalism in support of the claim:
\begin{center}
\emph{The presence of ideal objects \(in particular infinitesimals\) in Nonstandard Analysis yields the \emph{ubiquitous presence} of computational content.}
\end{center}
We are hopeful we have convinced the reader of this claim \emph{for the case where `computational content' has a mainstream mathematics meaning}.  
In other words, we believe we have effectively refuted Connes' critique by showing that theorems of (some kind of) computable mathematics can be obtained from theorems of Nonstandard Analysis, and vice versa.  

\medskip

The case of Bishop's critique is different as our results, while rich
in computational content, were all obtained in classical mathematics, thus not really constructive mathematics.    
Nonetheless, we have obtained `indispensable' constructive content from Nonstandard Analysis in Sections \ref{CTT} and \ref{HTT} in the form of modulus functions.   One can probably endlessly debate the status of a theorem of constructive mathematics proved via classical means.  We shall therefore take the middle ground and hope that Bishop would have acknowledged the ubiquitous presence of \emph{some kind of computational content} in Nonstandard Analysis (as he has done for Brouwer's intuitionistic mathematics and the Russian school of recursive mathematics), with the reservation that this content is not exactly what he has in mind in his branch of constructivism.  

\medskip

As to criticism of reverse formalism, perhaps the easiest and most straightforward way of dismissing the latter is based on the rejection of proof translations (as the latter are used in an essential way to prove Theorem \ref{TERM} in \cite{brie, sambon}).  We shall discuss this argument at length in Section \ref{promi} while we refer to \cite{kohlenbach3} for an introduction to proof translations, also called `functional interpretations'.  

\medskip

This discussion has wider ranging implications as follows:  proof translations provide a method of converting proofs of certain theorems in a given `strong' system into a proof in a given `weak' system, often adding extra computational information to the theorem in the process.  An interesting question emerges: 
\begin{quote}\emph{
Given a proof $p$ of a theorem $T$ in classical logic, and a proof translation $\Xi$ formulated in constructive mathematics and outputting proofs in the latter; in what way is the resulting theorem and proof $\Xi(p, T)$ `constructive'?   
}\end{quote}
We do not claim to answer this question definitively, but we do provide a number of arguments in Section \ref{promi} against the dismissal of reverse formalism based on the rejection of proof translations.

\subsection{Criticism of reverse formalism}\label{promi}
We discuss possible a criticism of reverse formalism, the proof translation from \cite{brie} needed to prove Theorem \ref{TERM} in particular.  
We do not claim to provide definite answers, but we do provide a number of arguments against the rejection of proof translations in our context. 
\subsubsection{An emphasis on informal mathematics}\label{informo}
An easy way to dismiss reverse formalism is to point out its reliance on formal systems (the proof translations in \cite{brie} in particular) and that Bishop was against the use of the latter in his branch of constructivism.    
We now argue against this easy dismissal as follows.  

\medskip

First of all, while Bishop has indeed sneered at formal systems (See Section~\ref{hotgie} for two telling quotes), he did change his mind later in life.  
Let us first consider the following historical remark by Nerode from the proceedings of Bishop's memorial meeting (\cite{bishf}).   
\begin{quote}
After the publication of his book \underline{Constructive Analysis}, Bishop made a
tour of the eastern universities that included Cornell. He told me then that he
was trying to communicate his viewpoint directly to the mathematical 
community, rather than through the logicians. He associated the logicians with
defending the turf of codified formal systems, while he himself believed in the
free exercise of positive affirmative mathematical faculties, free of artificial
formal limitations. After the eastern tour was over, he said the trip may
have been counterproductive. He felt that his mathematical audiences were
not taking the work seriously. He was surprised to get a more sympathetic
hearing from the logicians. (\cite{mem}*{p.\ 79}, underlining in original)
\end{quote}
After the aforementioned `sympathetic hearing' from logicians, Bishop indeed became more open to formal systems, as 
is clear from the following quote:
\begin{quote}
Another important foundational problem is to find a formal system that
will efficiently express existing predictive mathematics. I think we should
keep the formalism as primitive as possible, starting with a minimal system
and enlarging it only if the enlargement serves a genuine mathematical need.
In this way the formalism and the mathematics will hopefully interact to
the advantage of both. (\cite{nukino}*{p.\ 60})
\end{quote}
Hence, the blanket statement `Bishop was against the use of formal systems' does not do justice to history; the previous quote implies it is plainly wrong.  

\medskip

Secondly, \emph{even if} we follow the early Bishop in his rejection of formal systems trying to capture his constructive mathematics, the latter is \emph{not} the goal 
of reverse formalism or proof translations in general.  The only claim being made is that one can obtain via the proof translations certain results in constructive mathematics.  
As it happens, proof translations and Bishop's later views on constructive mathematics are compatible, and even intimately related, as follows: Bishop discusses in \cite{nukino}*{p.\ 56} the concept of \emph{numerical implication} an alternative constructive notion of implication based on G\"odel's Dialectica interpretation (the latter is also the basis for the proof translation in \cite{brie} and the proof of Theorem \ref{TERM}).  

\medskip

Bishop notes that in practice the usual definition of implication amounts to numerical implication.   
Furthermore, Bishop conjectures that numerical implication can be derived constructively, while his derivation in \cite{nukino} uses non-constructive principles like \emph{independence of premises} and \emph{Markov's principle}.  As it happens, system $\H$ contains nonstandard versions of the latter two axioms (See Definition \ref{flah}) and these axioms are essential in proving the \emph{Characterization Theorem} (\cite{brie}*{Theorem 5.8}) for the proof translation of $\H$ \emph{in the case of implication}.   

\medskip

While the previous is highly suggestive, we will not push things further than observing that Bishop's later views (esp.\ the notion of so-called numerical implication) were strongly inspired by proof translations.  

\subsubsection{An emphasis on proofs}
As it is clear from the definition of the BHK interpretation in Section \ref{hitch}, constructive mathematics has a certain emphasis on proofs.  
This is also clear from the following quotes.    
\begin{quote}
The basic tenets [of intuitionism] may be summarized as follows. [\dots]
(b) It does not make sense to think of truth or falsity of a mathematical
statement independently of our knowledge concerning the statement. A
statement is \emph{true} if we have proof of it, and \emph{false} if we can show that the
assumption that there is a proof for the statement leads to a contradiction.
For an arbitrary statement we can therefore not assert that it is either true
or false. (\cite{troeleke1}*{I.1.6})
\end{quote}
\begin{quote}
From an intuitionistic standpoint, therefore, an understanding of a mathematical statement consists in the capacity to recognize a proof of it when presented with one; and the truth of such a statement can consist only in the
existence of such a proof. (\cite{dummy}*{p.\ 5})
\end{quote}
As another case in point, Bridges' paper \emph{Constructive truth in practice} (\cite{brikkenschijten}) is devoted (mostly) to Bishop's mathematics and discusses 
provability at length, but does not even mention the word truth except in the title.  

\medskip

Now, we do not wish to imply that there is anything \emph{wrong} with this focus on proof; it is in fact a natural consequence of the constructivist philosophy of mathematics.  
However, in our opinion, the constructive preoccupation with \emph{proof} seems difficult to square with a rejection of proof translations.  
Rather, one would expect constructivists to embrace the latter as a `logical next step', again in our opinion.  

\medskip

In conclusion, we have observed that there is a focus on proofs in constructive mathematics.  In this light, the rejection of proof translations by constructivists seems odd, to say the least.     

\subsubsection{An emphasis on the mystical}
This section is perhaps the most controversial of this paper.   Our message is however simple:  Bishop faulted Brouwer for his metaphysical writings, while a rejection of certain results obtained by proof translations -in our opinion- has the danger of going off into metaphysical speculation again.  

\medskip
\noindent
Firstly, consider the following quote by Bishop on Brouwer's intuitionism.  
\begin{quote}
More important, Brouwer's system itself had traces of
idealism and, worse, of metaphysical speculation. There was a 
preoccupation with the philosophical aspects of constructivism at the 
expense of concrete mathematical activity.  (\cite{bish1}*{p.\ 6})
\end{quote}
It cannot be denied that Brouwer's work (\cite{brouw}) contains metaphysical speculation and elements of mysticism (See also \cites{hesselich, besselich}).  
We do not wish to judge Brouwer or intuitionism on this basis, but merely point out this fact and Bishop's view on this matter.  
We now reconsider the question:  
\begin{quote}\emph{
Given a proof $p$ of a theorem $T$ in classical logic, and a proof translation $\Xi$ formulated in constructive mathematics and outputting proofs in the latter; in what way is the resulting theorem and proof $\Xi(p, T)$ `constructive'?   
}\end{quote}
Now, it has been suggested that certain extensions of Heyting arithmetic like \textsf{HA}$^\omega$ (See e.g.\ \cite{diedif}) may serve as formalisations for Bishop's mathematics.  
Hence, proofs in Heyting arithmetic $\textsf{HA}^{\omega}$ produced by a proof interpretation which may be formulated in (a suitable extension of) Heyting arithmetic \emph{should} count as constructive mathematics, \emph{even if} the original proof (the input of the proof translation; a finite list of symbols) made use of classical logic.  

\medskip

There are no \emph{formal} objections one can bring against the previous claim, as the output of the proof translation as well as its verification take place in a formal system for constructive mathematics.  There are perhaps other objections one can bring against this claim, but our reply would be that there is the danger of going off into metaphysical speculation \emph{which was exactly what Bishop faulted Brouwer for and tried to avoid}.

\begin{ack}\rm
This research was supported by the following entities: FWO Flanders, the John Templeton Foundation, the Alexander von Humboldt Foundation, LMU Munich (via the Excellence Initiative), and the Japan Society for the Promotion of Science.  
The author expresses his gratitude towards these institutions. 
The author would also like to thank the two referees for their helpful remarks which greatly improved this paper.  
\end{ack}


\appendix

\section{The formal systems $\P$ and $\H$ in full detail}\label{FULL}


\subsection{G\"odel's system ${T}$}\label{TITI}
In this section, we briefly introduce G\"odel's system ${T}$ and the associated systems $\textsf{E-PA}^{\omega}$ and $\textsf{E-PA}^{\omega*}$.  
In his famous \emph{Dialectica} paper (\cite{godel3}), G\"odel defines an interpretation of intuitionistic arithmetic into a quantifier-free calculus of functionals.  This calculus is now known as `G\"odel's system ${T}$', and is essentially just primitive recursive arithmetic (\cite{buss}*{\S1.2.10}) with the schema of recursion expanded to \emph{all finite types}.  
The set of all finite types $\pmb{T}$ is:
\begin{center}
(i) $0\in \pmb{T}$   and   (ii)  If $\sigma, \tau\in \pmb{T}$ then $( \sigma \di \tau) \in \pmb{T}$,
\end{center}
where $0$ is the type of natural numbers, and $\sigma\di \tau$ is the type of mappings from objects of type $\sigma$ to objects of type $\tau$.  Hence, G\"odel's system ${T}$ includes `recursor' constants $\pmb{R}^{\rho}$ for every finite type $\rho\in \pmb{T}$, defining primitive recursion as follows:  
\be\label{PR}\tag{\textsf{\textup{PR}}}
\pmb{R}^{\rho}(f,g, 0):=f   \textup{ and } \pmb{R}^{\rho}(f, g, n+1):=g(n, \pmb{R}^{\rho}(f, g, n)),
\ee
for $f^{\rho}$ and $g^{0\di( \rho\di \rho)}$.  
The system $\textsf{E-PA}^{\omega}$ is a combination of \emph{Peano Arithmetic} and system $T$, and the full axiom of extensionality \eqref{EXT}.  
The detailed definition of $\textsf{E-PA}^{\omega}$ may be found in \cite{kohlenbach3}*{\S3.3};  We do introduce the notion of equality and extensionality in $\textsf{E-PA}^{\omega}$, as these notions are needed below.
\bdefi[Equality]\label{FAK}
The system $\textsf{E-PA}^{\omega}$ includes equality between natural numbers `$=_{0}$' as a primitive.  Equality `$=_{\tau}$' for type $\tau$-objects $x,y$ is then:
\be\label{aparth}
[x=_{\tau}y] \equiv (\forall z_{1}^{\tau_{1}}\dots z_{k}^{\tau_{k}})[xz_{1}\dots z_{k}=_{0}yz_{1}\dots z_{k}]
\ee
if the type $\tau$ is composed as $\tau\equiv(\tau_{1}\di \dots\di \tau_{k}\di 0)$.  
The usual inequality predicate `$\leq_{0}$' between numbers has an obvious definition, and the predicate `$\leq_{\tau}$' is just `$=_{\tau}$' with `$=_{0}$' replaced by `$\leq_{0}$' in \eqref{aparth}.    
The \emph{axiom of extensionality} is the statement that for all $\rho, \tau\in \pmb{T}$, we have:
\be\label{EXT}\tag{\textsf{E}}  
(\forall  x^{\rho},y^{\rho}, \varphi^{\rho\di \tau}) \big[x=_{\rho} y \di \varphi(x)=_{\tau}\varphi(y)   \big], 
\ee 
\edefi
Next, we introduce $\textsf{E-PA}^{\omega*}$, a definitional extension of $\textsf{E-PA}^{\omega}$ from \cite{brie} with a type for finite sequences.  In particular,  the set $\pmb{T}^{*}$ is defined as:
\begin{center}
(i) $0\in \pmb{T}^{*}$,   (ii)  If $\sigma, \tau\in \pmb{T}^{*}$ then $ (\sigma \di \tau) \in \pmb{T}^{*}$, and (iii) If $\sigma \in \pmb{T}^{*}$, then $\sigma^{*}\in \pmb{T}^{*}$,
\end{center}
where $\sigma^{*}$ is the type of finite sequences of objects of type $\sigma$.  The system $\textsf{E-PA}^{\omega*}$ includes \eqref{PR} for all $\rho\in \pmb{T}^{*}$, as well as dedicated `list recursors' to handle finite sequences for any $\rho^{*}\in \pmb{T}^{*}$.        
A detailed definition of $\textsf{E-PA}^{\omega*}$ may be found in \cite{brie}*{\S2.1}.  
We now introduce some notations specific to $\textsf{E-PA}^{\omega*}$, as also used in \cite{brie}.
\begin{nota}[Finite sequences]\label{skim}\rm
The system $\textsf{E-PA}^{\omega*}$ has a dedicated type for `finite sequences of objects of type $\rho$', namely $\rho^{*}$.  Since the usual coding of finite sequences of natural numbers goes through in $\textsf{E-PA}^{\omega*}$, we shall not always distinguish between $0$ and $0^{*}$. 
Similarly, we do not always distinguish between `$s^{\rho}$' and `$\langle s^{\rho}\rangle$', where the former is `the object $s$ of type $\rho$', and the latter is `the sequence of type $\rho^{*}$ with only element $s^{\rho}$'.  The empty sequence for the type $\rho^{*}$ is denoted by `$\langle \rangle_{\rho}$', usually with the typing omitted.  Furthermore, we denote by `$|s|=n$' the length of the finite sequence $s^{\rho^{*}}=\langle s_{0}^{\rho},s_{1}^{\rho},\dots,s_{n-1}^{\rho}\rangle$, where $|\langle\rangle|=0$, i.e.\ the empty sequence has length zero.
For sequences $s^{\rho^{*}}, t^{\rho^{*}}$, we denote by `$s*t$' the concatenation of $s$ and $t$, i.e.\ $(s*t)(i)=s(i)$ for $i<|s|$ and $(s*t)(j)=t(|s|-j)$ for $|s|\leq j< |s|+|t|$. For a sequence $s^{\rho^{*}}$, we define $\overline{s}N:=\langle s(0), s(1), \dots,  s(N)\rangle $ for $N^{0}<|s|$.  
For a sequence $\alpha^{0\di \rho}$, we also write $\overline{\alpha}N=\langle \alpha(0), \alpha(1),\dots, \alpha(N)\rangle$ for \emph{any} $N^{0}$.  By way of shorthand, $q^{\rho}\in Q^{\rho^{*}}$ abbreviates $(\exists i<|Q|)(Q(i)=_{\rho}q)$.  Finally, we shall use $\underline{x}, \underline{y},\underline{t}, \dots$ as short for tuples $x_{0}^{\sigma_{0}}, \dots x_{k}^{\sigma_{k}}$ of possibly different type $\sigma_{i}$.          
\end{nota}
We have used $\textsf{E-PA}^{\omega}$ and $\textsf{E-PA}^{\omega*}$ interchangeably in this paper.  Our motivation is the `star morphism' used in Robinson's approach to Nonstandard Analysis, and the ensuing potential for confusion.

\subsection{The classical system $\P$}\label{PAPA}
In this section, we introduce the system $\P$, a conservative extension of $\textsf{E-PA}^{\omega}$ with fragments of Nelson's $\IST$.  

\medskip

To this end, we first introduce the base system $\textsf{E-PA}_{\st}^{\omega*}$.  
We use the same definition as \cite{brie}*{Def.~6.1}, where \textsf{E-PA}$^{\omega*}$ is the definitional extension of \textsf{E-PA}$^{\omega}$ with types for finite sequences as in \cite{brie}*{\S2}.  
The set $\T^{*}$ is defined as the collection of all the terms in the language of $\textsf{E-PA}^{\omega*}$.    
\bdefi\label{debs}
The system $ \textsf{E-PA}^{\omega*}_{\st} $ is defined as $ \textsf{E-PA}^{\omega{*}} + \T^{*}_{\st} + \textsf{IA}^{\st}$, where $\T^{*}_{\st}$
consists of the following axiom schemas.
\begin{enumerate}
\item The schema\footnote{The language of $\textsf{E-PA}_{\st}^{\omega*}$ contains a symbol $\st_{\sigma}$ for each finite type $\sigma$, but the subscript is essentially always omitted.  Hence $\T^{*}_{\st}$ is an \emph{axiom schema} and not an axiom.\label{omit}} $\st(x)\wedge x=y\di\st(y)$,
\item The schema providing for each closed\footnote{A term is called \emph{closed} in \cite{brie} (and in this paper) if all variables are bound via lambda abstraction.  Thus, if $\underline{x}, \underline{y}$ are the only variables occurring in the term $t$, the term $(\lambda \underline{x})(\lambda\underline{y})t(\underline{x}, \underline{y})$ is closed while $(\lambda \underline{x})t(\underline{x}, \underline{y})$ is not.  The second axiom in Definition \ref{debs} thus expresses that $\st_{\tau}\big((\lambda \underline{x})(\lambda\underline{y})t(\underline{x}, \underline{y})\big)$ if $(\lambda \underline{x})(\lambda\underline{y})t(\underline{x}, \underline{y})$ is of type $\tau$.  We essentially always omit lambda abstraction for brevity.\label{kootsie}} term $t\in \T^{*}$ the axiom $\st(t)$.
\item The schema $\st(f)\wedge \st(x)\di \st(f(x))$.
\end{enumerate}
The external induction axiom \textsf{IA}$^{\st}$ is as follows.  
\be\tag{\textsf{IA}$^{\st}$}
\Phi(0)\wedge(\forall^{\st}n^{0})(\Phi(n) \di\Phi(n+1))\di(\forall^{\st}n^{0})\Phi(n).     
\ee
\edefi
Secondly, we introduce some essential fragments of $\IST$ studied in \cite{brie}.  
\bdefi[External axioms of $\P$]~
\begin{enumerate}
\item$\HAC_{\INT}$: For any internal formula $\varphi$, we have
\be\label{HACINT}
(\forall^{\st}x^{\rho})(\exists^{\st}y^{\tau})\varphi(x, y)\di \big(\exists^{\st}F^{\rho\di \tau^{*}}\big)(\forall^{\st}x^{\rho})(\exists y^{\tau}\in F(x))\varphi(x,y),
\ee
\item $\textsf{I}$: For any internal formula $\varphi$, we have
\[
(\forall^{\st} x^{\sigma^{*}})(\exists y^{\tau} )(\forall z^{\sigma}\in x)\varphi(z,y)\di (\exists y^{\tau})(\forall^{\st} x^{\sigma})\varphi(x,y), 
\]
\item The system $\P$ is $\textsf{E-PA}_{\st}^{\omega*}+\textsf{I}+\HAC_{\INT}$.
\end{enumerate}
\end{defi}
Note that \textsf{I} and $\HAC_{\INT}$ are fragments of Nelson's axioms \emph{Idealisation} and \emph{Standard part}.  
By definition, $F$ in \eqref{HACINT} only provides a \emph{finite sequence} of witnesses to $(\exists^{\st}y)$, explaining its name \emph{Herbrandized Axiom of Choice}.   

\medskip

The system $\P$ is connected to $\textsf{E-PA}^{\omega}$ by the following theorem.    
Here, the superscript `$S_{\st}$' is the syntactic translation defined in \cite{brie}*{Def.\ 7.1}.  
\begin{thm}\label{consresult}
Let $\Phi(\tup a)$ be a formula in the language of \textup{\textsf{E-PA}}$^{\omega*}_{\st}$ and suppose $\Phi(\tup a)^\Sh\equiv\forallst \tup x \, \existsst \tup y \, \varphi(\tup x, \tup y, \tup a)$. If $\Delta_{\intern}$ is a collection of internal formulas and
\be\label{antecedn}
\P + \Delta_{\intern} \vdash \Phi(\tup a), 
\ee
then one can extract from the proof a sequence of closed\footnote{Recall the definition of closed terms from \cite{brie} as sketched in Footnote \ref{kootsie}.\label{kootsie2}} terms $t$ in $\mathcal{T}^{*}$ such that
\be\label{consequalty}
\textup{\textsf{E-PA}}^{\omega*} + \Delta_{\intern} \vdash\  \forall \tup x \, \exists \tup y\in \tup t(\tup x)\ \varphi(\tup x,\tup y, \tup a).
\ee
\end{thm}
\begin{proof}
Immediate by \cite{brie}*{Theorem 7.7}.  
\end{proof}
The proofs of the soundness theorems in \cite{brie}*{\S5-7} provide an algorithm $\mathcal{A}$ to obtain the term $t$ from the theorem.  In particular, these terms 
can be `read off' from the nonstandard proofs.  The translation $S_{\st}$ can be formalised in any reasonable\footnote{Here, a `reasonable' system is one which can prove the usual properties of finite lists, for which the presence of the exponential function suffices.} system of constructive mathematics.  In fact, the formalisation of the results in \cite{brie} 
in the proof assistant Agda (based on Martin-L\"of's constructive type theory \cite{loefafsteken}) is underway in \cite{EXCESS}.      

\medskip

In light of the results in \cite{sambon}, the following corollary (which is not present in \cite{brie}) is essential to our results.  Indeed, the following corollary expresses that we may obtain effective results as in \eqref{effewachten} from any theorem of Nonstandard Analysis which has the same form as in \eqref{bog}.  It was shown in \cite{sambon, samzoo, samzooII} that the scope of this corollary includes the Big Five systems of Reverse Mathematics and the associated `zoo' (\cite{damirzoo}).  
\begin{cor}\label{consresultcor}
If $\Delta_{\intern}$ is a collection of internal formulas and $\psi$ is internal, and
\be\label{bog}
\P + \Delta_{\intern} \vdash (\forall^{\st}\underline{x})(\exists^{\st}\underline{y})\psi(\underline{x},\underline{y}, \underline{a}), 
\ee
then one can extract from the proof a sequence of closed$^{\ref{kootsie2}}$ terms $t$ in $\mathcal{T}^{*}$ such that
\be\label{effewachten}
\textup{\textsf{E-PA}}^{\omega*} + \Delta_{\intern} \vdash (\forall \underline{x})(\exists \underline{y}\in t(\underline{x}))\psi(\underline{x},\underline{y},\underline{a}).
\ee
\end{cor}
\begin{proof}
Clearly, if for internal $\psi$ and $\Phi(\underline{a})\equiv (\forall^{\st}\underline{x})(\exists^{\st}\underline{y})\psi(x, y, a)$, we have $[\Phi(\underline{a})]^{S_{\st}}\equiv \Phi(\underline{a})$, then the corollary follows immediately from the theorem.  
A tedious but straightforward verification using the clauses (i)-(v) in \cite{brie}*{Def.\ 7.1} establishes that indeed $\Phi(\underline{a})^{S_{\st}}\equiv \Phi(\underline{a})$.  
\end{proof}
For the rest of this paper, the notion `normal form' shall refer to a formula as in \eqref{bog}, i.e.\ of the form $(\forall^{\st}x)(\exists^{\st}y)\varphi(x,y)$ for $\varphi$ internal.  

\medskip

Finally, the previous theorems do not really depend on the presence of full Peano arithmetic.  
We shall study the following subsystems.   
\bdefi~
\begin{enumerate}
\item Let \textsf{E-PRA}$^{\omega}$ be the system defined in \cite{kohlenbach2}*{\S2} and let \textsf{E-PRA}$^{\omega*}$ 
be its definitional extension with types for finite sequences as in \cite{brie}*{\S2}. 
\item $(\QFAC^{\rho, \tau})$ For every quantifier-free internal formula $\varphi(x,y)$, we have
\be\label{keuze}
(\forall x^{\rho})(\exists y^{\tau})\varphi(x,y) \di (\exists F^{\rho\di \tau})(\forall x^{\rho})\varphi(x,F(x))
\ee
\item The system $\RCAo$ is $\textsf{E-PRA}^{\omega}+\QFAC^{1,0}$.  
\end{enumerate}
\edefi
The system $\RCAo$ is the `base theory of higher-order Reverse Mathematics' as introduced in \cite{kohlenbach2}*{\S2}.  
We permit ourselves a slight abuse of notation by also referring to the system $\textsf{E-PRA}^{\omega*}+\QFAC^{1,0}$ as $\RCAo$.
\begin{cor}\label{consresultcor2}
The previous theorem and corollary go through for $\P$ and $\textsf{\textup{E-PA}}^{\omega*}$ replaced by $\P_{0}\equiv \textsf{\textup{E-PRA}}^{\omega*}+\T_{\st}^{*} +\HAC_{\INT} +\textsf{\textup{I}}+\QFAC^{1,0}$ and $\RCAo$.  
\end{cor}
\begin{proof}
The proof of \cite{brie}*{Theorem 7.7} goes through for any fragment of \textsf{E-PA}$^{\omega{*}}$ which includes \textsf{EFA}, sometimes also called $\textsf{I}\Delta_{0}+\textsf{EXP}$.  
In particular, the exponential function is (all what is) required to `easily' manipulate finite sequences.    
\end{proof}

\subsection{The constructive system $\H$}
In this section, we define the system $\H$, the constructive counterpart of $\P$. 
The system $\textsf{H}$ was first introduced in \cite{brie}*{\S5.2}, and constitutes a conservative extension of Heyting arithmetic $\textup{\textsf{E-HA}}^{\omega} $ by \cite{brie}*{Cor.\ 5.6}.
We now study the system $\H$ in more detail.  

\medskip

Similar to Definition \ref{debs}, we define $ \textsf{E-HA}^{\omega*}_{\st} $ as $ \textsf{E-HA}^{\omega{*}} + \T^{*}_{\st} + \textsf{IA}^{\st}$, where $\textsf{E-HA}^{\omega*}$ is just $\textsf{E-PA}^{\omega*}$ without the law of excluded middle.  
Furthermore, we define
\[
\H\equiv \textup{\textsf{E-HA}}^{\omega*}_{\st}+\HAC + {\I}+\NCR+\textsf{HIP}_{\forall^{\st}}+\textsf{HGMP}^{\st},
\]
where $\HAC$ is $\HAC_{\INT}$ without any restriction on the formula, and where the remaining axioms are defined in the following definition.
\bdefi[Three axioms of $\H$]\label{flah}~
\begin{enumerate}\rm
\item $\textsf{HIP}_{\forall^{\st}}$
\[
[(\forall^{\st}x)\phi(x)\di (\exists^{\st}y)\Psi(y)]\di (\exists^{\st}y')[(\forall^{\st}x)\phi(x)\di (\exists y\in y')\Psi(y)],
\]
where $\Psi(y)$ is any formula and $\phi(x)$ is an internal formula of \textsf{E-HA}$^{\omega*}$. 
\item $\textsf{HGMP}^{\st}$
\[
[(\forall^{\st}x)\phi(x)\di \psi] \di (\exists^{\st}x')[(\forall x\in x')\phi(x)\di \psi] 
\]
where $\phi(x)$ and $\psi$ are internal formulas in the language of \textsf{E-HA}$^{\omega*}$.
\item \textsf{NCR}
\[
(\forall y^{\tau})(\exists^{\st} x^{\rho} )\Phi(x, y) \di (\exists^{\st} x^{\rho^{*}})(\forall y^{\tau})(\exists x'\in x )\Phi(x', y),
\]
where $\Phi$ is any formula of \textsf{E-HA}$^{\omega*}$
\end{enumerate}
\edefi
Intuitively speaking, the first two axioms of Definition \ref{flah} allow us to perform a number of \emph{non-constructive operations} (namely \emph{Markov's principle} and \emph{independence of premises}) 
on the standard objects of the system $\H$, provided we introduce a `Herbrandisation' as in the consequent of $\HAC$, i.e.\ a finite list of possible witnesses rather than one single witness. 
Furthermore, while $\H$ includes idealisation \textsf{I}, one often uses the latter's \emph{classical contraposition}, explaining why \textsf{NCR} is useful (and even essential) in the context of intuitionistic logic.  

\medskip

Surprisingly, the axioms from Definition \ref{flah} are exactly what is needed to convert nonstandard definitions (of continuity, integrability, convergence, et cetera) into the normal form $(\forall^{\st}x)(\exists^{\st}y)\varphi(x, y)$ for internal $\varphi$, as is clear from e.g.\ Section \ref{CTT}.
The latter normal form plays an equally important role in the constructive case as in the classical case by the following theorem.  
\begin{thm}\label{consresult2}
If $\Delta_{\intern}$ is a collection of internal formulas, $\varphi$ is internal, and
\be\label{antecedn3}
\textup{\textsf{H}} + \Delta_{\intern} \vdash \forallst \tup x \, \existsst \tup y \, \varphi(\tup x, \tup y, \tup a), 
\ee
then one can extract from the proof a sequence of closed terms $t$ in $\mathcal{T}^{*}$ such that
\be\label{consequalty3}
\textup{\textsf{E-HA}}^{\omega*} + \Delta_{\intern} \vdash\  \forall \tup x \, \exists \tup y\in \tup t(\tup x)\ \varphi(\tup x,\tup y, \tup a).
\ee
\end{thm}
\begin{proof}
Immediate by \cite{brie}*{Theorem 5.9}.  
\end{proof}
The proofs of the soundness theorems in \cite{brie}*{\S5-7} provide an algorithm $\mathcal{B}$ to obtain the term $t$ from the theorem.  
Finally, we point out one very useful principle to which we have access.  
\begin{thm}\label{doppi}
The systems $\P, \H$, and $\P_{0}$ prove \emph{overspill}, i.e.\
\be\tag{\textsf{OS}}
(\forall^{\st}x^{\rho})\varphi(x)\di (\exists y^{\rho})\big[\neg\st(y)\wedge \varphi(y)  \big],
\ee
for any internal formula $\varphi$.
\end{thm}
\begin{proof}
See \cite{brie}*{Prop.\ 3.3}.  
\end{proof}
In conclusion, we have introduced the systems $\H$, $\P$, which are conservative extensions of Peano and Heyting arithmetic with fragments of Nelson's internal set theory.  
We have observed that central to the conservation results (Corollary~\ref{consresultcor} and Theorem~\ref{consresult}) is the normal form $(\forall^{\st}x)(\exists^{\st}y)\varphi(x, y)$ for internal $\varphi$.  

\section{Nonstandard Analysis and qualitative information}\label{dva}
We list examples of applications of Nonstandard Analysis in which the latter is explicitly used to model qualitative phenomena.  
\begin{enumerate}
\item Raiman (\cites{raiman1, raiman2}) introduces a formal language \emph{FOG} for reasoning with qualitative notions `close', `negligible', and `same order of magnitude'.  Raiman proves \emph{FOG} to be validated in Robinson's \emph{Non-standard Analysis} (\cite{robinson1}).  The system \emph{FOG} has been used in economics and circuit design (\cites{raimanneke, baaiman})  
\item Weld studies the perturbation technique \emph{exaggeration} in \cite{welden} by means of Nonstandard Analysis.  In particular, he uses unlimited and infinitesimal values to study the limit behaviour for `large' and `small' parameter values. 
\item Davis presents a system based on Nonstandard Analysis for reasoning with qualitative notions like `small', `large', and `medium'. Results in dynamical systems and differential equations are obtained (\cite{davis13370})
\item Suenaga et al provide deductive verification framework of signals based on Nonstandard Analysis in \cites{seut}.  Rather than using approximations up to some large \emph{finite} precision, they use a correct-up-to-infinitesimals approximation using unlimited precision. 
\item Vopenka (\cite{vopsub}) proposes the use of various nonstandard structures (inside his \emph{alternative set theory} \textsf{AST}) to model vague phenomena.
\item Tzouvaras (\cite{toufke}) proposes the use of Nonstandard Analysis to model vague notions like `similarity' and `small'.  
\end{enumerate}
In light of the previous list, we hope the reader is convinced that Nonstandard Analysis \emph{can} be used to model qualitative notions.  
The usual caveat applies:  We do not claim that this modelling is the best or even accurate; we merely point out that people have used Nonstandard Analysis 
for this purpose in practice.  

\section{Some theorems relating to term extraction}
\begin{thm}\label{nogwelconsenal}
The system $\P$ proves that a normal form can be derived from an implication between normal forms. 
\end{thm}
\begin{proof}
Let $\varphi, \psi$ be internal and consider the following implication between normal forms:
\be\label{nora}
(\forall^{\st}x)(\exists^{\st}y)\varphi(x, y)\di (\forall^{\st}z)(\exists^{\st}w)\psi(z, w).  
\ee
Since standard functionals have standard output for standard input by Definition \ref{debs}, \eqref{nora} implies
\be\label{nora2}
(\forall^{\st}\zeta)\big[(\forall^{\st}x)\varphi(x, \zeta(x))\di (\forall^{\st}z)(\exists^{\st}w)\psi(z, w)\big].  
\ee
Bringing all standard quantifiers outside, we obtain the following normal form:
\be\label{nora3}
(\forall^{\st}\zeta, z)(\exists^{\st} w, x)\big[\varphi(x, \zeta(x))\di \psi(z, w)\big],
\ee
as the formula in square brackets is internal.
\end{proof}
It is an interesting exercise to establish the previous theorem for $\H$ in the stead of $\P$.  
\begin{thm}\label{hujiku} 
For internal $\varphi$, the system $\P$ proves that $(\forall \eps\approx 0)(\forall^{\st}x)(\exists^{\st}y^{\tau})\varphi(x, y, \eps)$ is equivalent to a normal form.   
\end{thm}  
\begin{proof}
Written out in full, the initial formula from the theorem is:
\[\textstyle
(\forall \eps)\big[(\forall^{\st} k^{0})(|\eps|<\frac{1}{k})\di (\forall^{\st}x)(\exists^{\st}y^{\tau})\varphi(x, y, \eps)\big],
\]
and bringing outside all standard quantifiers as far as possible:
\[\textstyle
(\forall^{\st}x)\underline{(\forall \eps)(\exists^{\st}y^{\tau}, k^{0})\big[|\eps|<\frac{1}{k}\di \varphi(x, y, \eps)\big]},
\]
the underlined formula is suitable for \emph{Idealisation}.  Applying the latter yields
\[\textstyle
(\forall^{\st}x)(\exists^{\st}w^{0^{*}}, z^{\tau^{*}}){(\forall \eps)(\exists y^{\tau} \in z, k^{0}\in w)\big[|\eps|<\frac{1}{k}\di \varphi(x, y, \eps)\big]},
\]
and let $N^{0}$ be the maximum of all $w(i)$ for $i<|w|$.  We obtain:
\be\label{harny}\textstyle
(\forall^{\st}x)(\exists^{\st}z^{\tau^{*}}, N){(\forall \eps)(\exists y^{\tau}\in z)\big[|\eps|<\frac{1}{N}\di \varphi(x, y, \eps)\big]}.
\ee
Clearly, \eqref{harny} implies the initial formula from the theorem. 
\end{proof}
\section{Examples in Reverse Mathematics of the computational content of Nonstandard Analysis}\label{EXAKE}

\subsection{Theorems equivalent to $\ACA_{0}$}\label{AZA}
In this section, we study the \emph{monotone convergence theorem} $\MCT$, i.e.\ the statement that \emph{every bounded increasing sequence of reals is convergent}, which is equivalent to arithmetical comprehension $\ACA_{0}$ by \cite{simpson2}*{III.2.2}.
We prove an equivalence between a nonstandard version of $\MCT$ and a fragment of \emph{Transfer}.
From this nonstandard equivalence, we obtain an effective RM equivalence involving $\MCT$ and arithmetical comprehension.  

\medskip

Firstly, the nonstandard version of $\MCT$ (involving nonstandard convergence) is:
\be\label{MCTSTAR}\tag{\MCT$_{\textsf{ns}}$}
(\forall^{\st} c_{(\cdot)}^{0\di 1})\big[(\forall n^{0})(c_{n}\leq c_{n+1}\leq 1)\di (\forall N,M\in \Omega)[c_{M}\approx c_{N}]    \big], 
\ee
where `$(\forall K\in \Omega)(\dots)$' is short for $(\forall K^{0}))(\neg\st(K)\di \dots)$.  
The effective version \MCT$_{\textsf{ef}}(t)$:
\be\label{MCTSTAR22}\textstyle
(\forall c_{(\cdot)}^{0\di 1},k^{0})\big[(\forall n^{0})(c_{n}\leq c_{n+1}\leq 1)\di (\forall N,M\geq t(c_{(\cdot)})(k))[|c_{M}- c_{N}|\leq \frac{1}{k} ]   \big].
\ee
We require two equivalent (\cite{kohlenbach2}*{Prop.\ 3.9}) versions of arithmetical comprehension: 
\be\label{mu}\tag{$\mu^{2}$}
(\exists \mu^{2})\big[(\forall f^{1})( (\exists n)f(n)=0 \di f(\mu(f))=0)    \big],
\ee
\be\label{mukio}\tag{$\exists^{2}$}
(\exists \varphi^{2})\big[(\forall f^{1})( (\exists n)f(n)=0 \asa \varphi(f)=0   ) \big],
\ee
Clearly, $(\exists^{2})$ (and therefore $(\mu^{2})$) is the functional version of $\ACA_{0}$.  
We also recall the restriction of Nelson's axiom \emph{Transfer} as follows:
\be\tag{$\paai$}
(\forall^{\st}f^{1})\big[(\forall^{\st}n^{0})f(n)\ne0\di (\forall m)f(m)\ne0\big].
\ee
Denote by $\textsf{MU}(\mu)$ the formula in square brackets in \eqref{mu}.  We have the following theorem which establishes the explicit equivalence between $(\mu^{2})$ and uniform $\MCT$.  
\begin{thm}\label{sef}
We have $\P\vdash \MCT_{\ns}\asa \paai$.  From this proof, terms $s, u$ can be extracted such that $\textup{\textsf{E-PA}}^{\omega*}$ proves:
\be\label{frood}
(\forall \mu^{2})\big[\textsf{\MU}(\mu)\di \MCT_{\ef}(s(\mu)) \big] \wedge (\forall t^{1\di 1})\big[ \MCT_{\ef}(t)\di  \MU(u(t))  \big].
\ee
\end{thm}
\begin{proof}
See \cite{sambon}*{\S4.1}.  
\end{proof}
We point out \eqref{frood} is the `effective' version of the equivalence $\ACA_{0}\asa \MCT$; the former is obtained from the corresponding `nonstandard' equivalence $\paai\asa \MCT_{\ns}$.  
Note that the latter proof proceeds by contradiction.  

\medskip

Finally, while we did not emphasise this in Section \ref{RM}, Reverse Mathematics usually studies mathematical theorems \emph{formalised in second-order arithmetic}.  
The latter only involves natural numbers and sets thereof, i.e.\ continuous functions on the real numbers are \emph{indirectly} present in the form of \emph{codes} (See \cite{kohlenbach4}*{\S4}).   
Now, \eqref{frood} is clearly \emph{not} part of second-order arithmetic (as it involves objects of type two), but it is possible to obtain results in second-order arithmetic from \eqref{frood}, as explored in \cite{sambonIII}.  

\subsection{Theorems equivalent to $\ATR_{0}$ and $\FIVE$}\label{expl}
In this section, we study equivalences relating to $\ATR_{0}$ and $\FIVE$, the strongest Big Five systems.  
We shall work with the \emph{Suslin functional} $(S^{2})$, the functional version of $\FIVE$.  
\be\label{suske}
(\exists S^{2})(\forall f^{1})\big[   S(f)=_{0} 0 \asa (\exists g^{1})(\forall x^{0}) (f(\overline{g}x)\ne 0)\big]. \tag{$S^{2}$}
\ee
Feferman has introduced the following version of the Suslin functional (See e.g.\ \cite{avi2}).
\be\label{suske2}
(\exists \mu_{1}^{1\di 1})\big[(\forall f^{1})\big(    (\exists g^{1})(\forall x^{0}) (f(\overline{g}x)\ne 0)\di (\forall x^{0}) (f(\overline{\mu_{1}(f)}x)\ne 0)\big)\big], \tag{$\mu_{1}$}
\ee
where the formula in square brackets is $\MUO(\mu_{1})$.  We require another instance of \emph{Transfer}:
\be\tag{$\Paai$}
(\forall f^{1})\big[    (\exists g^{1})(\forall x^{0}) (f(\overline{g}x)\ne 0)\di  (\exists^{\st} g^{1})(\forall^{\st} x^{0}) (f(\overline{g}x)\ne 0)\big].
\ee
%
%
We first consider $\PST$, i.e.\ the statement that \emph{every tree with uncountably many paths has a non-empty perfect subtree}.   
As proved in \cite{simpson2}*{V.5.5}, we have $\PST\asa \ATR_{0}$ and a uniform version of $\PST$ is equivalent to the Suslin functional by \cite{yamayamaharehare}*{Theorem~4.4}.  
Now, $\PST$ has the following nonstandard and uniform versions. 
\begin{thm}[$\PST_{\ns}$]For all standard trees $T^{1}$, there is standard $P^{1}$ such that 
\[
(\forall f_{(\cdot)}^{0\di 1})(\exists f\in T)(\forall n)(f_{n}\ne_{1} f) \di \textup{$P$ is a non-empty perfect subtree of $T$}.
\]
\end{thm}
\begin{thm}[$\PST_{\ef}(t)$]
For all trees $T^{1}$, we have
\[
(\forall f_{(\cdot)}^{0\di 1})(\exists f\in T)(\forall n)(f_{n}\ne_{1} f) \di \textup{$t(T)$ is a non-empty perfect subtree of $T$}.
\]
\end{thm}
As a technicality, we require that $P$ as in the previous two principles consists of a pair $(P', p')$ such that $P'$ is a perfect subtree of $T$ such that $p'\in P'$.      
\begin{thm}\label{sef2334}
We have $\P\vdash \PST_{\ns}\asa \Paai$.  From the latter, terms $s, u$ can be extracted such that $\textup{\textsf{E-PA}}^{\omega*}$ proves:
\be\label{frood3}
(\forall \mu_{1})\big[\textsf{\MUO}(\mu_{1})\di \PST_{\ef}(s(\mu_{1})) \big] \wedge (\forall t^{1\di 1})\big[ \PST_{\ef}(t)\di  \MUO(u(t))  \big].
\ee
\end{thm}
\begin{proof}
See \cite{sambon}*{\S4.5}.  
\end{proof}
In conclusion, \eqref{frood3} is the `effective' version of \cite{yamayamaharehare}*{Theorem~4.4}; the former is obtained from the corresponding `nonstandard' equivalence $\Paai\asa \PST_{\ns}$.  
Note that the latter proof proceeds by contradiction.  

\medskip
%
Another \emph{more mathematical} statement which can be treated along the same lines is \emph{every countable Abelian group is a direct sum of a divisible and a reduced group}.  The latter is called $\DIV$ and equivalent to $\FIVE$ by \cite{simpson2}*{VI.4.1}.  By the proof of the latter, the reverse implication is straightforward;  We shall therefore study $\DIV\di \FIVE$.  

\medskip

To this end, let $\DIV(G, D, E)$ be the statement that the countable Abelian group $G$ satisfies $G=D\oplus E $, where $D$ is a divisible group and $E$ a reduced group.  
The nonstandard version of $\DIV$ is as follows:
\be\tag{$\DIV_{\ns}$}
(\forall^{\st}G)(\exists^{\st} D, d, E)\big[\DIV(G, D, E)\wedge (D\ne\{0_{G}\}\di d\in D)\big],
\ee
where we used the same technicality as for $\PST_{\ns}$.  The effective version is:
\be\tag{$\DIV_{\ef}(t)$}
(\forall G)\big[\DIV(G, t(G)(1), t(G)(2))\wedge (t(G)(1)\ne\{0_{G}\}\di t(G)(3)\in t(G)(1))\big].
\ee
We have the following theorem.  
\begin{thm}\label{sef23345}
We have have $\P\vdash\DIV_{\ns}\di \Paai$.  From the latter, a term $u$ can be extracted such that $\textup{\textsf{E-PA}}^{\omega*}$ proves:
\be\label{frood3555}
(\forall t^{1\di 1})\big[ \DIV_{\ef}(t)\di  \MUO(u(t))  \big].
\ee
\end{thm}
\begin{proof}
See \cite{samzoo}*{\S4.5}.  
\end{proof}

\bye